\documentclass[11pt,a4paper]{article}

\usepackage[a4paper,top=2.5cm,bottom=2.5cm,left=2.5cm,right=2.5cm]{geometry}

\usepackage{mathpazo}
\usepackage[T1]{fontenc}
\usepackage[utf8]{inputenc}
\usepackage{lmodern}    
\usepackage[dvipsnames]{xcolor}
\usepackage{soul}
\usepackage[normalem]{ulem}
\usepackage{cancel}
\usepackage{nicefrac}

\usepackage{setspace}
\usepackage{amsmath,amssymb,amsthm,mathtools}
\usepackage{dsfont}      
\usepackage{xifthen}
\usepackage{cancel}
\usepackage{gensymb}

\usepackage{tikz}
\usepackage{float}
\usepackage{subcaption}
\usepackage[section]{placeins}
\usepackage{pgfplots}
\pgfplotsset{compat=1.18}

\definecolor{fmi}{rgb}{0,0.48,0.55}
\definecolor{fsublue}{RGB}{0,102,204}
\definecolor{fmiDark}{rgb}{0,0.38,0.45} 
\definecolor{mydarkred}{RGB}{95,5,10}

\usepackage[backend=biber, style=numeric, natbib=true, maxbibnames=10]{biblatex}
\usepackage[autostyle=true]{csquotes}

\usepackage{hyperref}

\usepackage{enumitem}

\usepackage{titlesec}
\titleformat{\section}
  {\normalfont\Large\bfseries\centering}
  {\thesection}{1em}{}
\titleformat{\subsection}
  {\normalfont\normalsize\bfseries}
  {\thesubsection}{1em}{}

\usepackage{pgfplots}
\pgfplotsset{compat=1.18}
\usepackage{tikz-3dplot}
\tdplotsetmaincoords{70}{110}

\newtheorem{theorem}{Theorem}[section]
\newtheorem{proposition}[theorem]{Proposition}

\newtheorem{lemma}[theorem]{Lemma}
\theoremstyle{definition}
\newtheorem{defn}[theorem]{Definition}
\newtheorem{remark}[theorem]{Remark}
\newtheorem{example}[theorem]{Example}

\makeatletter
\renewenvironment{proof}[1][\proofname]{%
	\par\pushQED{\qed}%
	\normalfont\topsep6\p@\@plus6\p@\relax
	\trivlist
	\item[\hskip\labelsep\textbf{\textit{#1.}}]%
}{%
	\popQED\endtrivlist\@endpefalse
}
\makeatother

\newcommand{\R}{\mathbb{R}}

\newcommand{\norm}[1]{\left\lVert {#1} \right\rVert}

\renewcommand{\geq}{\geqslant}
\renewcommand{\leq}{\leqslant}

\DeclareMathOperator{\cl}{cl}
\let\int\relax
\DeclareMathOperator{\int}{int}
\DeclareMathOperator{\relint}{ri}

\DeclareMathOperator{\conv}{conv}
\DeclareMathOperator{\cone}{cone}
\DeclareMathOperator*{\bd}{bd}

\DeclareMathOperator*{\Min}{Min}
\DeclareMathOperator*{\wMin}{wMin}

\DeclareMathOperator*{\verti}{vert}
\DeclareMathOperator*{\extdir}{extdir}

\DeclareUnicodeCharacter{1EC7}{\d{\^{e}}}
\DeclareUnicodeCharacter{1EE5}{\d{u}}

\begin{document}

\title{A Solution Concept for Convex Vector Optimization Problems based on a User-defined Region of Interest}

\author{Daniel D{\"o}rfler\thanks{daniel.doerfler@uni-jena.de, Friedrich Schiller University Jena, Germany. ORCiD: 0000-0002-9503-3619.}
  \and Rebecca K{\"o}hler\thanks{rebecca.koehler@uni-jena.de, Friedrich Schiller University Jena, Germany. ORCiD: 0009-0009-5713-6471.}
  \and Andreas L{\"o}hne\thanks{andreas.loehne@uni-jena.de, Friedrich Schiller University Jena, Germany. ORCiD: 0000-0003-0872-4735.}}

\date{\today}

\maketitle

\begin{abstract}
\setlength{\parskip}{0pt}
\noindent
This work addresses arbitrary convex vector optimization problems, which constitute a general framework for multi-criteria decision-making in diverse real-world applications. Due to their complexity, such problems are typically tackled using polyhedral approximation. Existing solution concepts rely on additional assumptions, such as boundedness, polyhedrality of the ordering cone, or existence of interior points in the ordering cone, and typically focus on absolute error measures. We introduce a solution concept based on the homogenization of the upper image that employs relative error measures and avoids additional structural assumptions. Although minimality is not explicitly required, a form of approximate weak minimality is implicitly ensured under attainability. The concept is straightforward, requiring only a single precision parameter and, owing to relative errors, remains robust under scaling of the target functions. Homogenization also eliminates the need for the binary distinction between points far from the origin and directions which can lead to numerical difficulties. Furthermore, in practice decision-makers often identify a region where preferred solutions are expected. Our concept supports both a global overview of the upper image and a refined local perspective within such a user-defined region of interest (RoI). We present a decision-making procedure enabling iterative refinement of this region and the associated preferences.

 \par
\end{abstract}

\section{Introduction}
In multi-criteria decision-making the goal is to identify efficient solutions — decisions where improving one objective necessarily deteriorates another. \textit{Convex vector optimization problems}, characterized by convex objective functions and feasible region, arise naturally in fields such as engineering, economics, financial mathematics, and dynamic programming, see e.g. \cite{jahn2011vector}, \cite{rudloff2021certainty} and \cite{kovavcova2021time}. Due to their inherent complexity, these problems are typically approached via polyhedral approximation \cite{ehrgott2011approximation, lohne2014primal, ararat2022norm, wagner2023algorithms}, which aims to represent the set of feasible outcomes in a tractable form for decision support.\par
Existing solution concepts often impose further structural assumptions on the problem setting. The majority of literature proposes solution concepts for bounded problems, see, e.g., \cite{ehrgott2011approximation} and \cite{lohne2014primal}. Moreover, most approaches use absolute error measures \cite{ehrgott2011approximation, lohne2014primal, ararat2022norm, wagner2023algorithms} such as the \textit{Hausdorff distance} to quantify approximation quality. While effective in \textit{bounded} or \textit{self-bounded} settings, these methods encounter limitations when applied to general \textit{unbounded} problems \cite{ulus2018tractability}. To address these limitations, the $(\varepsilon,\delta)$-concept was independently proposed in \cite{wagner2023algorithms} and \cite{dorfler2022approximation}. It is a hybrid approach combining absolute and relative error measures. In the context of convex vector optimization, \cite{wagner2023algorithms} apply the concept via a two-step process. It begins with approximating the recession cone of the upper image using a bound $\delta$ on the relative, \textit{truncated Hausdorff distance}. Based on this, an approximation within a bounded region emerges in the second step. This approximation does not exceed a specified absolute Hausdorff distance $\varepsilon$ to the upper image. Notably, this region is not predefined but results implicitly from the parameter setting and structure of the upper image. Although the method addresses the general unbounded case, it still relies on certain structural properties. While in \cite{wagner2023algorithms} polyhedrality of the ordering cone is assumed for the sake of algorithmic implementation, existence of interior points in the ordering cone is an inherent requirement of the concept.\par
In this work, we introduce an alternative solution concept for general convex vector optimization problems based on \textit{homogeneous $\delta$‑approximation}. It relies solely on a relative error measure, avoids additional structural assumptions, treats points and directions in a unified manner, and supports a practical decision-making process that can be focused on a user-specified \textit{region of interest (RoI)}. The concept is straightforward, requiring only a single precision parameter $\delta$. While we do not enforce (approximate) minimality, proximity to \textit{weakly minimal elements} can still be bounded under feasibility within a certain region depending on the parameter $\delta$.\par
In many applications, decision-makers identify a region where preferred solutions are expected – either based on prior knowledge or interactively during the decision process. We formalize this intuition via a user-defined RoI, in which tighter approximation guarantees are provided. Outside this region, the approximation still captures global properties of the upper image. We propose an interactive procedure that enables iterative refinement of the RoI and the associated preferences, facilitating a preference-driven exploration of the upper image.\par
The use of the concept of \textit{homogenization} in approximately solving convex vector optimization problems is further motivated by the goal of reducing the upper image to the simpler structure of a convex cone, thereby facilitating its approximation and the implementation. Homogenization refers to a geometric concept in which the closed conic hull of a given nonempty convex set $Z \subseteq \mathbb{R}^m$ embedded in the level 1 of the space $\mathbb{R}^{m+1}$ is constructed. To be precise, the homogenization of $Z$ is given in \cite[p. 18 ff.]{brinkhuis2020convex} as follows:
$$\hom Z := \cl \cone (Z \times \{1\}).$$
In this way, a closed convex set, such as the upper image of a convex vector optimization problem, can be uniquely identified with a convex cone. A homogeneous $\delta$-approximation of the set is obtained by approximating its homogenization with another cone, intersecting this conic approximation with the embedding hyperplane, and projecting the result back to the original space by dropping the last component, see \cite[p. 18 ff.]{brinkhuis2020convex}. The illustration in Figure \ref{figHomog} gives a visual impression of the idea.

\begin{figure}[h!]
\centering
\begin{tikzpicture}[tdplot_main_coords, scale=2]
		
		\draw[->] (0,0,0) -- (2,0,0) node[anchor=north east]{};
		\draw[->] (0,0,0) -- (0,2,0) node[anchor=north west]{};
		\draw[->] (0,0,0) -- (0,0,1.5) node[anchor=south]{};

		\draw[fmi] (0,0,0) -- (0.97,0.51,1);
		\draw[fmi] (0,0,0) -- (0.9,1.3,1);
		\fill[fmi!20, opacity=0.4] 
		(0,0,0) -- (0.97,0.51,1) -- (0.9,1.3,1) -- cycle;
		
		\fill[gray!20,opacity=0.6] (0,0,1) -- (1.5,0,1) -- (1.5,1.6,1) -- (0,1.6,1) -- cycle;
		\draw[gray] (0,0,1) -- (1.5,0,1) -- (1.5,1.6,1) -- (0,1.6,1) -- cycle;
		
		\fill[gray!20,opacity=0.6] (0,0,0) -- (1.5,0,0) -- (1.5,1.6,0) -- (0,1.6,0) -- cycle;
		\draw[gray] (0,0,0) -- (1.5,0,0) -- (1.5,1.6,0) -- (0,1.6,0) -- cycle;
		
		\draw[dashed] (0,0,1) -- (0,0,0);
		\node at (-0.1,-0.1,1) {\small $1$};
		
		\filldraw[fill=fmi, draw=fmi, opacity=0.3] (0.9,0.9,1) circle[radius=0.4];
		\node[fmi] at (1.5,1.1,1.25) {\tiny{$Z\times \{1\}$}};
		
		\node[fmi] at (1.5,1.1,0.75) {\tiny{$\hom Z$}};
\end{tikzpicture}
\caption{Homogenization of a convex set $Z$.}
\label{figHomog}
\end{figure}
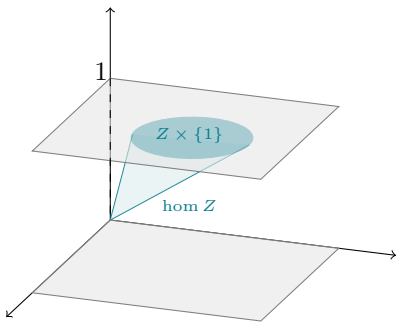

Points and directions of the original set are represented by the same geometric objects in its homogenization. This avoids the numerical difficulties of distinguishing between distant points and recession directions when approximating a set.\par
In this work, we consider convex vector optimization problems of the general form

\begin{align}
\min_{x} F(x) \quad 
\text{s.t.} \quad x \in S \quad \text{w.r.t.} \quad \leq_C \tag{VCP}
\end{align}
where $S \subseteq \mathbb{R}^n$ is a nonempty convex set, $F: \mathbb{R}^n \rightarrow \mathbb{R}^m$ is a $C-$convex, vector-valued function, and $C \subseteq \mathbb{R}^m$ is a non-trivial pointed closed convex ordering cone. \par
The remainder of this work is structured as follows. Section \ref{prelim} introduces notation and basic concepts. The main results are given in Section \ref{main}, where the homogeneous $\delta$-solution concept and its theoretical and practical implications are presented. Section \ref{DMprocedure} presents a straightforward decision-making procedure and Section \ref{comparison} compares the concept with existing approaches. In Section \ref{proofs} we provide the mathematical proofs for the results of Section \ref{main} and in Section \ref{outlook}, we conclude with final remarks.
\section{Preliminaries} \label{prelim}
Let $Z \subseteq \mathbb{R}^m$ be a nonempty convex set. We denote the interior, relative interior, closure, and boundary of a set $Z$ by $\int Z$, $\relint Z$, $\cl Z$, and $\bd Z$, respectively. Open and closed Euclidean $r$-balls around a point $p \in \mathbb{R}^m$ are given by $U_r(p)$ and $B_r(p)$, respectively, and the Euclidean unit ball around the origin is denoted by $\mathbb{B}$. By $\left\lVert x \right\rVert$, we denote the Euclidean norm of a vector $x \in \mathbb{R}^m$. By $d(x,y) = \lVert x-y \rVert$ we denote the Euclidean distance of two points $x,y \in \mathbb{R}^m$.
A cone $K$ is defined by $\lambda k \in K$ for all $k \in K$ and for all $\lambda \geq 0$. The convex and conic hull of $Z$ are denoted by $\conv Z$ and $\cone Z$, respectively. The \textit{recession cone} of $Z$ is given by $0^+ Z = \left \{y \in \mathbb{R}^m \: \middle| \:   x + \lambda y \in Z \text{ for all } x \in Z, \lambda \geq 0 \right \}$, with its elements being referred to as \textit{recession directions} or \textit{directions} of $Z$.\par
A polyhedron is the intersection of finitely many closed half-spaces \cite{tyrrell1970convex}. By the Minkowski-Weyl theorem, any polyhedron $P$ with $\verti P \neq \emptyset $ can be written as $P = \conv(\verti P) + \cone(\extdir P)$, where $\verti P$ and $\extdir P $ are the sets of extreme points and directions of $P$, see e.g. \cite[Chapter 8]{schrijver1998theory}.\par
A closed convex cone is said to be \textit{solid} if it has nonempty interior, \textit{pointed} if it does not contain any lines and \textit{non-trivial} if it is not equal to $\{0\}$. The non-trivial pointed closed convex ordering cone $C$ defines an ordering in (VCP) via the relation $y_1 \leq_C y_2$ if and only if $y_2 - y_1 \in C$ for all $y_1, y_2 \in \mathbb{R}^m$. For (VCP), we say that the function $F$ is \textit{$C-$convex} if 
$$\forall x_1,x_2 \in \mathbb{R}^n \; \forall \lambda \in (0,1): \; F(\lambda x_1 + (1-\lambda) x_2) \leq_C \lambda F(x_1) + (1-\lambda) F(x_2).$$
The set of \textit{$C-$minimal elements} of the problem (VCP) is equal to the set of $C-$minimal elements of $F[S]$ given by $\Min F[S] = \left \{ y \in F[S] \: \middle| \:   (\left \{y \right\} - C \setminus \left\{0 \right\}) \cap F[S] = \emptyset \right\}$ and \textit{weakly $C$-minimal elements} of (VCP) are analogously given by $\wMin F[S] = \left\{ y \in F[S] \: \middle| \:   (\{y\} - \int C) \cap F[S] = \emptyset \right\}$. Note that the notion of weak minimality is only meaningful if the ordering cone $C$ is solid. A feasible point $x \in S$ is a (weak) minimizer of the problem (VCP) if $F(x) \in \Min F[S]$ ($F(x) \in \wMin F[S]$), see e.g. \cite[Chapter 1]{eichfelderadaptive}.\par
We define the \textit{upper image} of a problem of the form (VCP) by $\mathcal{P} = \cl(F[S] + C)$. We call problem (VCP) bounded with respect to a cone $K \supseteq C$ if $\mathcal{P} \subseteq B + K$ for some compact set $B \subseteq \mathbb{R}^m$. Note that if $K$ is solid, this definition of boundedness with respect to $K$ is equivalent to $\mathcal{P} \subseteq \{p\} + K$ for some $p \in \mathbb{R}^m$ \cite[Theorem 3.2]{dorfler2022approximation}. If the problem is bounded with respect to $C$, the problem itself is called bounded; otherwise, it is unbounded, see e.g. \cite[Definition 4.2]{ulus2018tractability}. A problem is self-bounded if $\mathcal{P} \neq \mathbb{R}^m$ and it is bounded with respect to $0^+ \mathcal{P}$ \cite[Definition 4.5]{ulus2018tractability}.\par
To assess the quality of approximations, we use Euclidean distances, the Hausdorff distance and the truncated Hausdorff distance. The Euclidean distance between a point $y \in \mathbb{R}^m$ and a nonempty set $M \subseteq \mathbb{R}^m$ is defined as $d(y,M) = \inf_{z \in M} \lVert y - z \rVert$. The Hausdorff distance between two nonempty convex sets $M_1$ and $M_2$ is defined by \cite[p. 12]{hiriart2004fundamentals}
$$d_H(M_1, M_2) = \max\left\{\sup_{x \in M_1} d(x, M_2), \sup_{x \in M_2} d(x, M_1)\right\}.$$
For closed, convex cones $K_1$ and $K_2$, the truncated Hausdorff distance is defined as \cite{iusem2010distances}
$$d_{tH}(K_1, K_2) = d_H(K_1 \cap \mathbb{B}, K_2 \cap \mathbb{B}).$$

\section{A Solution Concept Based on Homogenization} \label{main}
With the novel solution concept, we aim to address general convex vector optimization problems of the form (VCP), explicitly allowing for unboundedness. Our objective of approximating the upper image of the problem is motivated by several practical considerations. Since decision-makers usually base their decisions on the outcomes of feasible choices rather than the choices themselves, and the dimension of the outcome space is typically lower than that of the decision space, considerable effort has been devoted to representing or approximating the set of possible outcomes $F[S]$, see e.g. \cite{benson1998outer}. More recent literature has shifted focus toward approximating the upper image, see for instance \cite{hamel2014benson} for the linear case and \cite{lohne2014primal, ararat2022norm, keskin2023outer, wagner2023algorithms} for general convex settings. In contrast to $\mathcal{P}$, the set $F[S]$ is not necessarily convex in our framework. The advantage of approximating $\mathcal{P}$ rather than $F[S]$ is that the approximation of convex sets becomes significantly more tractable. This is due to their favourable structural properties that allow convex polyhedra with their simple, well-understood geometry to serve as effective approximants \cite{bronstein2008approximation}. Convex sets admit both primal and dual representations which can be exploited in the construction of such polyhedral approximations \cite[Chapter 4.2]{brinkhuis2020convex}. Importantly, passing from $F[S]$ to the upper image does not entail any loss of information regarding $C$-minimal elements as the closure of the set $F[S]$ and $\mathcal{P}$ share the same $C$-minimal elements. Consequently, the upper image serves as a descriptive and informative object, offering the decision-maker the most valuable insight into the set of attainable outcomes.\par
We use the approach of homogeneous $\delta$-approximation of the upper image, which is a means of approximating arbitrary convex sets using polyhedra and a relative error measure \cite{dorfler2024polyhedral}. We now provide a formal definition of a homogeneous $\delta$-approximation of a nonempty convex set.

\begin{defn}\label{defHomDelApprox}(cf. \cite[Definition 3.2]{dorfler2024polyhedral}). 
    Let $\delta>0$. A nonempty closed convex set $P \subseteq \mathbb{R}^m$ is called a homogeneous $\delta$-approximation of a nonempty convex set $Z \subseteq \mathbb{R}^m$ if
    $$d_{tH}(\hom P, \hom Z) \leq \delta.$$
\end{defn}

Using this type of approximation comes with several advantages. First, this measure of distance between two convex sets $P$ and $Z$ is a metric for nonempty closed convex sets \cite[Theorem 4.46]{rockafellar1998variational}. We strive to obtain an approximation $P$ that is sufficiently close to the set $Z$ in this metric. Another aspect is that the error parameter $\delta$ specifies the maximum Hausdorff distance on the unit ball $\mathbb{B}$ between the homogenizations of set $Z$ and that of its approximation $P$. This is a relative error bound since it allows the absolute distance to grow as we consider points further away from the origin. For $x \in \mathbb{R}^m$ with $d(x,Z) \leq \varepsilon$, it holds that \cite{dorfler2024polyhedral}
\[
d\left(\frac{(x,1)^T}{\norm{(x,1)^T}}, \hom Z\right) \leq \frac{\varepsilon}{\sqrt{x^Tx+1}}
\]
with the upper bound being sharp. Hence, applied to the context of convex vector optimization, we obtain a solution concept that exhibits a certain robustness with respect to scaling of the objective functions. Another advantage of using homogeneous $\delta$-approximation is that we can employ the simpler, conic structure of the homogenizations. In particular, points and directions of a set are represented by the same objects in its homogenization. To this end, consider the following fundamental characterization given in \cite{tyrrell1970convex}:

\begin{proposition}\label{rock} \cite[Theorem 8.2]{tyrrell1970convex}. Let $Z$ be a nonempty closed convex set. Then
$$\hom Z = \cone(Z \times \{1\}) \cup (0^+ Z \times \{0\}).$$
\end{proposition}

The representation in Proposition \ref{rock} shows that the homogenization cone stores information on both points $z$ of $Z$ as $(z,1)^T$ and directions $d$ of $0^+Z$ as $(d,0)^T$. Approximating the homogenization cone means not enforcing a strict distinction between points with large norm and directions. For an unbounded set $Z$, any direction $d\in 0^+Z$ can be approximated by points $z + t d \in Z$ for arbitrary $z \in Z$ and $t$ large. For $t\to\infty$, the homogenized points $(z + t d,\, 1)^T$ scaled by $1/t$ converge to $(d,0)^T$ in $\hom Z$:
	\[
	\frac{1}{t}(z + t d,\, 1)^T = \left(d + \frac{z}{t},\, \frac{1}{t}\right)^T \xrightarrow{t\to\infty} (d,0)^T.
	\]
This shows that directions of $Z$ arise as limits of points with increasing norm. Consequently, when approximating $Z$ via the approximation of its homogenization, a direction of $Z$ may be approximated by a point that lies sufficiently far from the origin. \par
In the context of convex vector optimization, we define a homogeneous $\delta$-solution of a problem of the form (VCP) in the following way:

\begin{defn} \label{DefSolConc} Let $\delta>0$. For a problem of the form (VCP), a finite nonempty set $X \subseteq S$ is called a homogeneous $\delta$-solution with \textit{approximate upper image} $P_X := \conv F[X]+C$ if
\begin{equation*}
d_{tH}(\hom P_X, \hom \mathcal{P}) \leq \delta.
\end{equation*}
\end{defn}

This concept constitutes a straightforward application of homogeneous $\delta$-approximations to the context of convex vector optimization. The approximate upper image is an inner homogeneous $\delta$-approximation of the upper image.\par
In contrast to established solution concepts, the approach employs Euclidean distances as a sole measure of proximity between the upper image and its approximation instead of relying on notions of ($\varepsilon$-)minimality (see, for example, \cite{gutierrez2006approximate} or \cite{ehrgott2011approximation} for some notions of $\varepsilon$-minimality). We propose that a decision-maker is primarily concerned with understanding the overall structure of the upper image, with more detailed interpretation around individual points being a secondary consideration. To this end, the Euclidean distance sufficiently captures the proximity of the approximate upper image to the upper image. If the decision-maker wishes to identify $C-$minimal elements or approximately $C-$minimal elements, a refinement step using scalarization techniques can be applied, see e.g., \cite[Chapter 2]{eichfelderadaptive}.\par
In the upcoming part, a term for an upper bound on the absolute approximation error within a given region of interest (RoI), represented by a Euclidean ball around the origin  with chosen radius $r$, is derived. To be precise, given a homogeneous $\delta$-solution with approximate upper image $P_X$, we derive a term $\alpha(r,\delta)$ with $d(y, \bd \mathcal{P}) \leq \alpha(r,\delta)$ for $y \in B_r(0) \cap \bd P_X$. Note that shifting the upper image accordingly enables the decision-maker to choose another centre than the origin for their RoI (see Remark \ref{rem:translation}). Figure \ref{del0.1} shows how the graph of the error term is presented to the decision-maker. The points on the graph represent the upper bound $\alpha(r,\delta)$ on the absolute error (vertical axis) within the RoI with radius $r$ (horizontal axis) given a fixed value of the precision parameter $\delta$. Based on the behaviour of this error term for a given choice of $\delta$, the decision-maker attempts to specify a RoI in which the approximation error $\alpha(r,\delta)$ can be bounded by a satisfying value. If the decision-maker obtains a satisfactory point on the graph, the given value of $\delta$ can be used for the homogeneous $\delta$-approximation. Otherwise, they are presented with an adjusted graph for smaller values of $\delta$, until a satisfying trade-off is obtained on the graph. It should also be taken into consideration that reducing the value of $\delta$ increases the computational effort for the approximation.\par
The error bound $\alpha(r,\delta)$ tends to infinity as the radius $r$ approaches the value $R_\delta$ in Figure \ref{del0.1}. Thus the error bound is only valid within a certain \textit{region of validity}, which is a ball of radius $R_\delta$ around the origin. The size of the region of validity $R_\delta$ is dependent on the precision parameter $\delta$ and increases as $\delta$ decreases.

\begin{figure}[H]
	\centering
\begin{tikzpicture}[scale = 0.7]
	\begin{axis}[
		width=\textwidth,
		height=6cm,
		grid=major,
		xlabel={$r$},
		ylabel={$\alpha(r,\delta)$},
		xmin=0, xmax=10.5,
		ymin=0, ymax=85,
		title={$\delta = 0{.}1$},
		domain=0.01:10,
		samples=1000,
		]
		
		\addplot [
		domain=0.01:9.9,
		samples=1000,
		line width=2pt,
		fmi
		] {(0.1*(x^2 + 1))/(sqrt(1 - 0.1^2) - 0.1*x)};
		\draw[dashed, black] (axis cs:9.9499,0) -- (axis cs:9.9499,100);
		\node[black, anchor=south east] at (axis cs:9.9499,5) {\large{$R_\delta$}};
		
	\end{axis}
\end{tikzpicture}
	\caption{Error bound $\alpha(r,\delta)$ within the region of interest depending on $r$ for $\delta = 0.1$. $R_\delta$ marks the radius for the region of validity for the error bound.}
	\label{del0.1}
\end{figure}
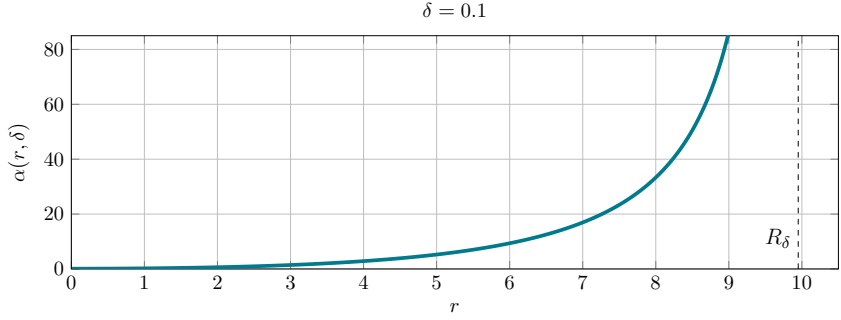

Given the selected value of $\delta$, a corresponding approximation of the upper image is generated that gives the desired precision in the RoI, as well as a global overview beyond that. The decision-maker can iteratively explore the properties of the upper image by deciding to shift or enlarge their RoI or change the tolerated error bound within this region. More details on the decision-making procedure, as well as an example, are shown in Section \ref{DMprocedure}.\par
We want to further motivate the concept with its wide applicability. Theorem \ref{homExists} states that under no additional requirements a homogeneous $\delta$-solution always exists. For the proofs of the results in the remainder of this section, we refer the reader to Section \ref{proofs}.

\begin{theorem} \label{homExists}
    Let a problem of the form (VCP) be given. Then, for every $\delta > 0$, a homogeneous $\delta$-solution of (VCP) exists.
\end{theorem}

We now turn to the relationship between the truncated Hausdorff distance of the homogenized sets and the error bound for the distance between their de‑homogenized counterparts. For a problem (VCP), let $X$ be a homogeneous $\delta$-solution for some $\delta \in (0,1)$ with approximate upper image $P_X = \conv F[X] + C$. We will assume $\delta \in (0,1)$ throughout the rest of this work as the two excluded cases, $\delta=0$ and $\delta = 1$, do not require further consideration: If $\delta = 0$, a homogeneous $\delta$-solution would yield an exact representation $P_X = \mathcal{P}$, which does not exist in general. If $\delta = 1$, the homogeneous $\delta$-approximation is not meaningful since $d_{tH}(K_1,K_2) \leq 1$ holds true for arbitrary cones $K_1$ and $K_2$ by the definition of the truncated Hausdorff distance. \par
We analyze boundary points $y$ of the approximate upper image with respect to their proximity to boundary points of the upper image. Every boundary
point $n \in \bd\mathcal{P}$ is a weakly $C-$minimal element of (VCP) if $n \in F[S]$. In the following, this attainability condition is referred to as feasibility. It is implied by the characterization in Proposition \ref{rock} that for points $y$ that are sufficiently far from the origin, i.e. $\norm{y}$ large enough, it is no longer guaranteed that the minimal truncated Hausdorff distance is attained with respect to a boundary point of the upper image rather than a direction of its recession cone. Consequently, we obtain a region of validity for the error bound on the distance to a boundary point of the upper image presented in Theorem \ref{absch}. We define it as an open Euclidean ball around the origin with radius
$$R_{\delta }:=\sqrt{\frac{1}{\delta ^2}-1}.$$

Theorem \ref{absch} presents the upper bound $\alpha(r,\delta)$ on the Euclidean distance between any point $y \in \bd P_X \cap B_r(0)$ with $0 < r < R_\delta$ and a boundary point of the upper image of (VCP), corresponding to a weakly $C-$minimal element of (VCP) under feasibility. To this end, $B_r(0)$ is interpreted as the RoI for the decision-maker. The bound $\alpha(r, \delta)$ can be expressed explicitly as a function of $r$ and $\delta$.

\begin{theorem}\label{absch} Let $X$ be a homogeneous $\delta$-solution of (VCP) for $\delta \in (0,1)$ with approximate upper image $P_X = \conv F[X] + C$. Let $0<r<R_\delta$ and $y \in \bd P_X \cap B_r(0)$. Then the absolute approximation error in $y$ can be bounded in the following way:

$$d(y,\bd\mathcal{P}) \leq \frac{\delta(r^2+1)}{\sqrt{1-\delta^2}-\delta r} = : \alpha(r,\delta).$$
\end{theorem}

\begin{remark}
    From the proof of Theorem \ref{absch}, one can directly construct a worst‑case example in which $d(y,\bd \mathcal{P})= \alpha(r,\delta)$ under the stated assumptions. Let $\delta = 0.6$, $r = 1$, and for a problem (VCP), $S= [-7,-1]$, $C = \R_+$ and $F(x) = x$. Then $X = \{-1\}$ is a homogeneous $\delta$-solution of (VCP) since
    $$\mathcal P = [-7, \infty) \subseteq \R, \qquad P_X = F[X] + C = [-1, \infty) \subseteq \R$$
    and $d_{tH}\left(\hom \mathcal P, \hom P_X\right) = 0.6$. For the points $y = -1 \in \bd P_X \cap B_r(0)$ and $n = -7 \in \bd \mathcal P$ we clearly have $d(y, \bd \mathcal P) = d(y,n) = 6$. Furthermore, $\alpha(1,0.6) = 6$. This demonstrates that the bound in Theorem \ref{absch} is sharp.
\end{remark}

\begin{remark}
	\label{rem:translation}
	Theorem \ref{absch} is stated for points $y \in \bd P_X \cap B_r(0)$. However, the result extends directly to balls centered at arbitrary points $p$ in the following way. If a homogeneous $\delta$-solution is computed for the translated upper image  $\mathcal{P} - \{p\}$, then Theorem \ref{absch} applies to points $y \in B_r(p)$ in the original setting by considering $y-p \in B_r(0)$ on the boundary of the approximation in the translated setting. This translation is used in Section \ref{DMprocedure} when working with user-defined regions of interest $B_r(p)$.
\end{remark}

We now provide an interpretation for boundary points of the approximate upper image that are not covered by the assumptions of Theorem \ref{absch}. Theorem \ref{oneOrTwo} shows that approximate boundary points $y \in \bd P_X$ outside the region of validity $U_{R_\delta}(0)$ are close to weakly $C-$minimal elements of (VCP) under feasibility or close to recession directions of the upper image. 

\begin{theorem} \label{oneOrTwo}
Let $X$ be a homogeneous $\delta$-solution of (VCP) for $\delta \in (0,1)$ with approximate upper image $P_X = \conv F[X] + C$. Let $y \in \bd P_X$. Then, one of the following holds true:
\begin{enumerate}[label=(\roman*)]
    \item \label{itm:i} There exists $n \in \bd \mathcal{P}$ such that
    \[
    q :=
    d\left(
    \frac{(n,1)^T}{\|(n,1)^T\|},
    \cone\{(y,1)^T\}
    \right)
    \leq \delta \quad \text{and} \quad
    d(y,\bd\mathcal{P}) \leq \alpha(\|y\|,q),
    \]
    with $\alpha(\|y\|,q)$ as defined in Theorem \ref{absch}.
    \item \label{itm: ii} $y \neq 0$ and $d\left(\frac{y}{\left\lVert y \right\rVert}, 0^+ \mathcal{P}\right) \leq \frac{\left\lVert (y, 1)^T \right\rVert}{\left\lVert y \right\rVert} \delta$.
\end{enumerate}
\end{theorem}\par

Hence, for any point $y$ on the boundary of the approximate upper image of a given problem (VCP), one of the following holds: Its Euclidean distance to the boundary of the upper image, and thereby to a weakly $C-$minimal element of (VCP) under feasibility, can be bounded by an error term, or its distance to some direction of the upper image can be bounded. Given an approximate boundary point $y \in B_r(0)$ with $0<r<R_\delta$, the first certainly holds true as shown in Theorem \ref{absch}. Approximate boundary points further from the origin cannot be strictly attributed to approximating boundary points or recession directions of the upper image. They do, however, have relative proximity to one or the other and thereby hold valuable information about the global behaviour of the upper image.
\section{Decision-Making Procedure}\label{DMprocedure}
We now aim to clarify how the preceding results, specifically the error bound derived in Theorem \ref{absch}, can be applied in practice. We consider a decision-maker who seeks a general overview of the upper image of a given (VCP) and is particularly interested in its behaviour around a specific point $p \in \mathbb{R}^m$.

The first part of the procedure will be to define the RoI which is given as a Euclidean ball $B_r(p)$ with user-specified point of reference $p$ and radius $r$. For this, the initial reference point $p$ is chosen in one of two ways. Either there is a prevailing preference of the decision-maker, or an arbitrary choice is made that can be iteratively refined during the decision-making procedure similar to the description in Example \ref{example} (see specifically step 1 visualized in Figure \ref{fig:step1}). Based on the choice of $p$, the translation specified in Remark \ref{rem:translation} is applied. Note that whenever a new reference point is chosen, this translation is applied before any approximation is computed. \par
Next, the decision-maker is presented with the graph of the error term $\alpha(r,\delta)$ derived in Theorem \ref{absch}. The graph can be evaluated for different, selected values of $\delta$ (see Figure \ref{deltasDM}). We note that the reference point $p$ is used for the localization of the RoI, while the precision within this region is quantified by the error term $\alpha(r,\delta)$. This error term is controlled via the approximation parameter $\delta$.
\pgfplotsset{compat=1.18}

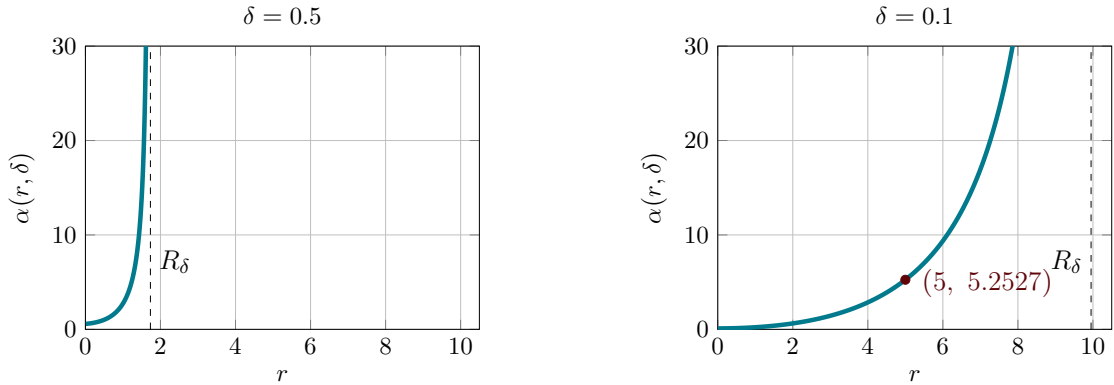
\begin{figure}[H]
	\centering
	
	\begin{minipage}{0.48\textwidth}
	\centering
	\begin{tikzpicture}[scale =0.85]
	\begin{axis}[
		width=\textwidth,
		height=6cm,
		grid=major,
		xlabel={$r$},
		ylabel={$\alpha(r,\delta)$},
		xmin=0, xmax=10.5,
		ymin=0, ymax=30,
		title={$\delta = 0{.}5$},
		domain=0.01:10,
		samples=1000,
		]
		\addplot [
		domain=0.01:1.73,
		samples=1000,
		line width=2pt,
		fmi
		] {(0.5*(x^2 + 1))/(sqrt(1 - 0.5^2) - 0.5*x)};
		\draw[dashed, black] (axis cs:1.732,0) -- (axis cs:1.732,30);
		\node[black, anchor=south west] at (axis cs:1.732,5) {\large{$R_\delta$}};
	\end{axis}
\end{tikzpicture}
	\end{minipage}
	\hfill
	\begin{minipage}{0.48\textwidth}
	\centering
\begin{tikzpicture}[scale=0.85]
	\begin{axis}[
		width=\textwidth,
		height=6cm,
		grid=major,
		xlabel={$r$},
		ylabel={$\alpha(r,\delta)$},
		xmin=0, xmax=10.5,
		ymin=0, ymax=30,
		title={$\delta = 0{.}1$},
		domain=0.01:10,
		samples=1000,
		]
		
		\addplot [
		domain=0.01:9.9,
		samples=1000,
		line width=2pt,
		fmi
		] {(0.1*(x^2 + 1))/(sqrt(1 - 0.1^2) - 0.1*x)};
		\draw[dashed, black] (axis cs:9.95,0) -- (axis cs:9.95,30);
		\node[black, anchor=south east] at (axis cs:9.95,5) {\large{$R_\delta$}};
		\addplot[
		only marks,
		mark=*,
		color=mydarkred
		] coordinates {(5, 5.2527)}; 
		
		\node[mydarkred, anchor=south west] at (axis cs:5.2, 2.5) {\large{$(5, \ 5.2527)$}};
		
	\end{axis}
\end{tikzpicture}
	\end{minipage}
	
	\caption{Error bound $\alpha(r,\delta)$ on the distance of boundary points of a homogeneous $\delta$-approximation to the boundary of the upper image (see Theorem \ref{absch}).
	The bound is given as a function of the radius $r$ for fixed values of large $\delta$ on the left and smaller $\delta$ on the right.
	The plots illustrate how the approximation error grows with the size of the RoI, with smaller $\delta$  yielding tighter bounds. The vertical dashed lines indicate the bound of the region of validity $R_\delta$ for each $\delta$, beyond which the error bound is undefined. As an example, the point on the right graph highlights a feasible trade-off between accuracy and coverage that is a possible choice for the decision-maker when setting $\delta = 0.1$. It means that, in an approximate upper image that is a homogeneous $0.1$-approximation, any boundary point in the RoI with radius 5 is in at most 5.2527 units distance of a weakly $C-$minimal element (given feasibility).
	}
	\label{deltasDM}
\end{figure} \par

The objective is to select the precision parameter $\delta$ as large as possible in order to minimize computational effort while simultaneously satisfying the decision-maker’s requirements. Notably, the choice of $\delta$ implicitly determines the radius for the region of validity $R_\delta>0$, which defines the extent beyond which no local error bound within a RoI can be guaranteed without adjusting $\delta$ (see Theorem \ref{absch}). Hence, for large values of $\delta$, less computational effort is necessary at the expense of a less defined picture (as can be seen by the steeper error curve on the left of Figure \ref{deltasDM}). The smaller the chosen value of $\delta$, the more comprehensive the representation of the upper image becomes. Once a satisfactory point $(r, \alpha)$ is selected from the presentation in Figure \ref{deltasDM}, a homogeneous $\delta$-solution is computed for the corresponding value of $\delta$. This approximation satisfies $d(y, \bd \mathcal{P}) \leq \alpha$ for any $y$ on the boundary of the approximation within the RoI $B_r(p)$.

The decision-maker is subsequently presented with the approximation and, if desired, a corresponding homogeneous $\delta$-solution of (VCP). This result provides a comprehensive overview of the upper image, enabling the identification of new points of interest. By adjusting either the reference point or the value of $\delta$ — guided by the visualization in Figure \ref{deltasDM} — the decision-maker can iteratively refine their exploration of the upper image in subsequent approximation rounds. \par
Provided that there exists at least one $C$-minimal element of (VCP), the decision-maker may select boundary points of interest from the final approximation, which can be refined to $C$-minimal elements by solving a suitable scalarization of (VCP), see, e.g., \cite[Chapter 2]{eichfelderadaptive}. This procedure also yields the pre-image of the resulting $C-$minimal elements, corresponding to the values of the decision variables required to attain the desired target. 
We now proceed to illustrate this process with a concrete example.

\begin{example}\label{example}
Let $\mathcal{P} :=  \{(x,y)^T \in \mathbb{R}^2 \mid \: x^2 - 1 \leq y \} + \mathbb{R}^2_+$ be the upper image of some problem (VCP).

\begin{figure}[H]
	\centering
	\begin{tikzpicture}[scale=0.55]
		\begin{axis}[
			width=10cm,
			height=10cm,
			xmin=-7.5, xmax=40,
			ymin=-3.5, ymax=40,
			xtick distance=15,
			ytick distance=15,
			xlabel={$x$},
			ylabel={$y$}
			]
			
			\draw[gray, thick, domain=-22:0, samples=100] plot (\x,{\x*\x - 1});
			
			\draw[gray, thick] (0,-1) -- (60,-1);
			
			\fill[gray!40]
			(-10,99) --
			plot[domain=-10:0, samples=100] (\x,{\x*\x - 1}) --
			(0,-1) -- (50,-1) -- (50,50) -- cycle;
			
			\node[gray, font=\bfseries] at (rel axis cs:0.5,0.5) {\Large{\(\mathcal{P}\)}};
			
		\end{axis}
	\end{tikzpicture}
	\caption{Bounded portion of the unbounded upper image in Example \ref{example}.}
	\label{fig:upperEx}
\end{figure}
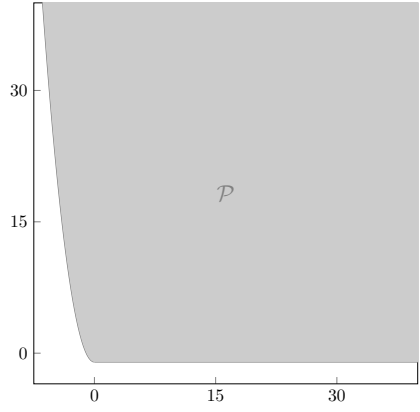

Naturally, the decision-maker cannot be presented with a visual representation like Figure \ref{fig:upperEx} since the upper image of a (VCP) is in general arbitrarily difficult to visualize. Hence, the goal is to obtain a tractable approximation. Let the decision-maker proceed as specified above. To locate the upper image and obtain a rough, initial impression, they select a relatively large value of $\delta$ together with an arbitrary reference point. For instance, consider the approximation with $\delta = 0.9$ within the RoI $B_{0.4}((0,-20)^T)$. In this case, the initial reference point is $p=(0,-20)^T$. Note that the intersection of the RoI and the set $\mathcal{P}$ is empty. Nevertheless, as illustrated in Figure \ref{fig:step1}, the homogeneous $\delta$-approximation provides a preliminary impression, thereby enabling the decision-maker to redefine the RoI in a more meaningful way.

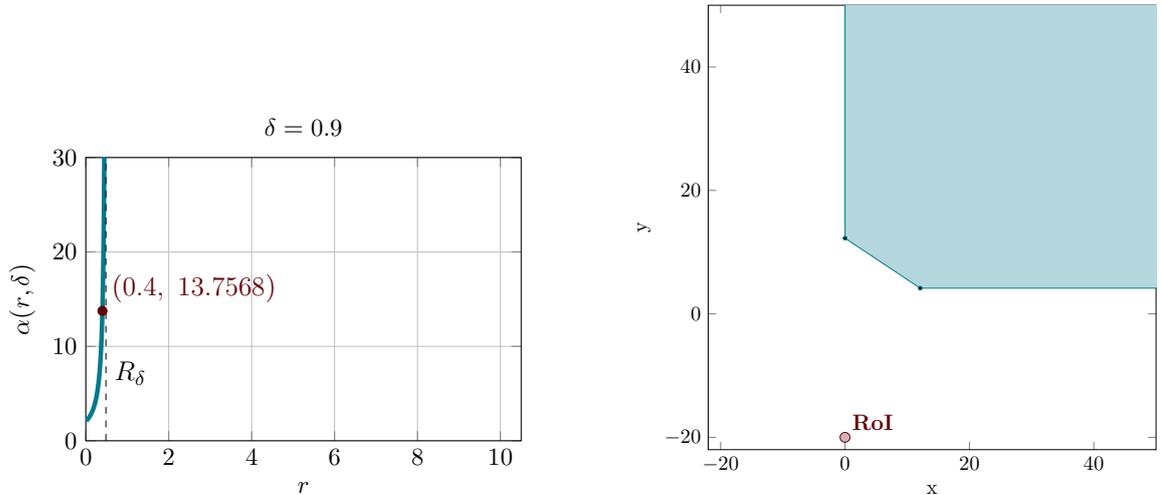
\begin{figure}[H]
	\begin{minipage}[t]{0.52\textwidth}
	\begin{subfigure}[t]{\textwidth}
	\centering

\begin{tikzpicture}[scale=0.85]
	\begin{axis}[
		width=\textwidth,
		height=6cm,
		grid=major,
		xlabel={$r$},
		ylabel={$\alpha(r,\delta)$},
		xmin=0, xmax=10.5,
		ymin=0, ymax=30,
		title={$\delta = 0{.}9$},
		domain=0.01:10,
		samples=1000,
		]
		
		\addplot [
		domain=0.01:0.484,
		samples=1000,
		line width=2pt,
		fmi
		] {(0.9*(x^2 + 1))/(sqrt(1 - 0.9^2) - 0.9*x)};
		\draw[dashed, black] (axis cs:0.484,0) -- (axis cs:0.484,100);
		\node[black, anchor=south west] at (axis cs:0.484,5) {\large{$R_\delta$}};
		\addplot[
		only marks,
		mark=*,
		color=mydarkred
		] coordinates {(0.4, 13.75677)}; 
		
		\node[mydarkred, anchor=south west] at (axis cs:0.4,13.75677) {\large{$(0.4,\ 13.7568)$}};
		
	\end{axis}
\end{tikzpicture}
	\caption{Chosen trade-off between size of the region of interest (RoI) ($r = 0.4)$ and upper bound on the error within that region ($\alpha(0.4,0.9) = 13.7568$) for given $\delta = 0.9$.}
	\end{subfigure}
	\end{minipage}\hfill
	\begin{minipage}[t]{0.44\textwidth}
	\begin{subfigure}[t]{\textwidth}
	\centering
	\begin{tikzpicture}[scale=0.7]
	\begin{axis}[
		xlabel={x},
		ylabel={y},
		width=10cm,
		height=10cm,
		xmin=-22, xmax=50,
		ymin=-22, ymax=50,
		xtick distance=20,
		ytick distance=20
		]
		
		\addplot[
		fill=fmi!30,
		draw=fmi,
		mark=*,
		mark size=1pt,
		mark options={fill=black}
		]
		table [x index=0, y index=1, col sep=comma] {Pictures/Vertices/STEP1_C.csv};
		
		\draw[draw=mydarkred, fill=mydarkred, fill opacity=0.3] 
		(axis cs:0,-20) circle (0.8)
		node[above right, text=mydarkred,opacity=1] {\large{\textbf{RoI}}};
		
	\end{axis}
\end{tikzpicture}
	\caption{Approximate upper image for $\delta = 0.9$ in the Example \ref{example} with centre of the approximation defined by the chosen RoI $B_{0.4}((0,-20)^T)$.}
	\end{subfigure}
	\end{minipage}\hfill
	\caption{Step 1 of the decision-making procedure. Locating the upper image and generating a rough impression using a high value for $\delta$ ($\delta = 0.9$) and an initial reference point ($p = (0,-20)^T$).}
	\label{fig:step1}
\end{figure}

Given this impression, the decision-maker may choose to shift their RoI. For instance, let them choose the reference point $(0,0)^T$ and a radius of $1.5$. With $\delta = 0.5$, Figure \ref{fig:step2} shows that the maximum error within the RoI is approximately $14$. Approximating the upper image with $\delta =0.5$ now yields a more refined representation, as illustrated on the right in Figure \ref{fig:step2}. As the boundary of the approximate upper image intersects the RoI, the error bound of Theorem \ref{absch} now applies to the corresponding boundary points. Moreover, substantial information about the global behaviour of the upper image is conveyed by vertices outside of the RoI.\par
Finally, if the decision-maker is satisfied with the location of their RoI, the approximation can be further refined, for example by choosing $\delta = 0.1$ and the RoI $B_{5}(0)$ (see Figure \ref{fig:step3}) or $\delta = 0.01$ and the RoI $B_{50}(0)$ (see, Figure \ref{fig:step4}). In Figure \ref{fig:step4}, a highly detailed depiction is obtained. Within the RoI, and especially close to the origin, the distance to weakly $C-$minimal elements is small under feasibility. Outside of the RoI, the global behaviour is revealed through vertices approximating recession directions of the upper image and vertices capturing the characteristic parabolic shape of the upper image on the left side.\par
When the decision-maker is content with the outcome of their exploration, they are in a position to make a well-informed choice. They can select points from the boundary of the approximate upper image and refine them to $C-$minimal or weakly $C-$minimal elements for the given problem, see e.g. \cite[Chapter 2]{eichfelderadaptive}.

\begin{figure}[H]
	\begin{minipage}[t]{0.52\textwidth}
	\begin{subfigure}[t]{\textwidth}
	\centering
\begin{tikzpicture}[scale = 0.85]
	\begin{axis}[
		width=\textwidth,
		height=6cm,
		grid=major,
		xlabel={$r$},
		ylabel={$\alpha(r,\delta)$},
		xmin=0, xmax=10.5,
		ymin=0, ymax=30,
		title={$\delta = 0{.}5$},
		domain=0.01:10,
		samples=1000,
		]
		\addplot [
		domain=0.01:1.73,
		samples=1000,
		line width=2pt,
		fmi
		] {(0.5*(x^2 + 1))/(sqrt(1 - 0.5^2) - 0.5*x)};
		\draw[dashed, black] (axis cs:1.732,0) -- (axis cs:1.732,30);
		\node[black, anchor=south west] at (axis cs:1.732,5) {\large{$R_\delta$}};
		\addplot[
		only marks,
		mark=*,
		color=mydarkred
		] coordinates {(1.5, 14.006)}; 
		
		\node[mydarkred, anchor=south west] at (axis cs:1.7, 14.006) {\large{$(1.5, 14.0056)$}};
		
	\end{axis}
\end{tikzpicture}
	\caption{Chosen trade-off between size of the region of interest (RoI) ($r = 1.5)$ and upper bound on the error within that region ($\alpha(1.5,0.5) = 14.0056$) for given $\delta = 0.5$.}
	\end{subfigure}
	\end{minipage}\hfill
	\begin{minipage}[t]{0.44\textwidth}
	\begin{subfigure}[t]{\textwidth}
	\centering
	\begin{tikzpicture}[scale = 0.7]
	\begin{axis}[
		xlabel={x},
		ylabel={y},
		width=10cm,
		height=10cm,
		xmin=-22, xmax=50,
		ymin=-22, ymax=50,
		xtick distance=20,
		ytick distance=20
		]
		
		\addplot[
		fill=fmi!30,
		draw=fmi,
		mark=*,
		mark size=1pt,
		mark options={fill=black}
		]
		table [x index=0, y index=1, col sep=comma] {Pictures/Vertices/STEP2_C.csv};
		
		\draw[draw=mydarkred, fill=mydarkred, fill opacity=0.3] 
		(axis cs:0,0) circle (1.8)
		node[below = 4mm, right = 2mm, text=mydarkred, opacity=1] {\large{\textbf{RoI}}};
		
	\end{axis}
\end{tikzpicture}
	\caption{Approximate upper image for $\delta = 0.5$ in the Example \ref{example} with centre of the approximation defined by the chosen RoI $B_{1.5}((0,0)^T)$.}
	\end{subfigure}
	\end{minipage}\hfill
	\caption{Step 2 of the decision-making procedure. Refining the choice of reference point ($p = (0,0)^T$) to obtain a global impression of the upper image using a slightly lower value of $\delta$ ($\delta = 0.5$).}
	\label{fig:step2}
\end{figure}
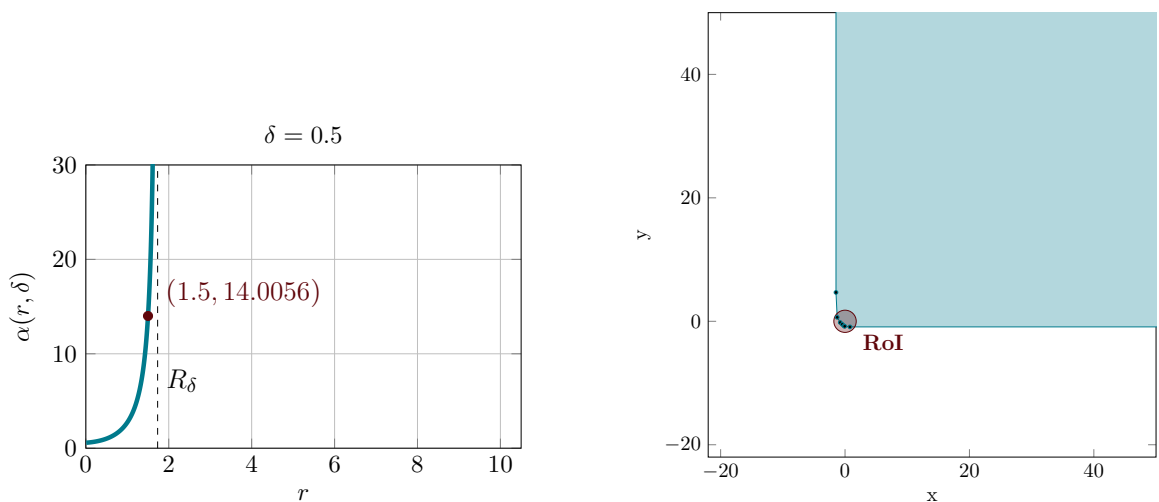
 
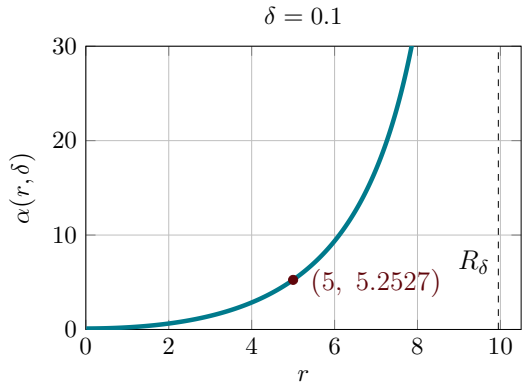
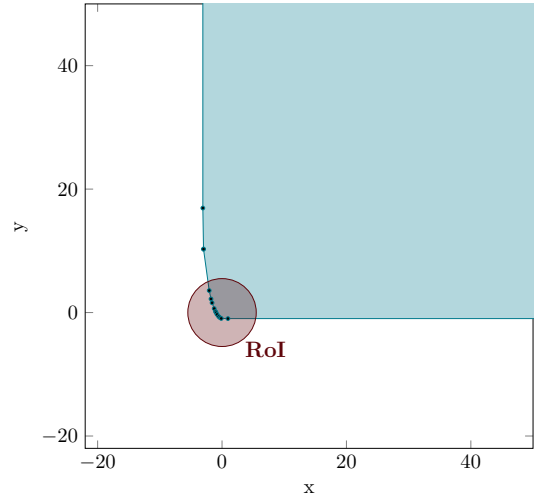
\begin{figure}[H]
	\begin{minipage}[t]{0.52\textwidth}
	\begin{subfigure}[t]{\textwidth}
	\centering
\begin{tikzpicture}[scale=0.85]
	\begin{axis}[
		width=\textwidth,
		height=6cm,
		grid=major,
		xlabel={$r$},
		ylabel={$\alpha(r,\delta)$},
		xmin=0, xmax=10.5,
		ymin=0, ymax=30,
		title={$\delta = 0{.}1$},
		domain=0.01:10,
		samples=1000,
		]
		
		\addplot [
		domain=0.01:9.9,
		samples=1000,
		line width=2pt,
		fmi
		] {(0.1*(x^2 + 1))/(sqrt(1 - 0.1^2) - 0.1*x)};
		\draw[dashed, black] (axis cs:9.95,0) -- (axis cs:9.95,30);
		\node[black, anchor=south east] at (axis cs:9.95,5) {\large{$R_\delta$}};
		\addplot[
		only marks,
		mark=*,
		color=mydarkred
		] coordinates {(5, 5.2527)}; 
		
		\node[mydarkred, anchor=south west] at (axis cs:5.2, 2.5) {\large{$(5, \ 5.2527)$}};
		
	\end{axis}
\end{tikzpicture}
	\caption{Chosen trade-off between size of the region of interest (RoI) ($r = 5)$ and upper bound on the error within that region ($\alpha(5,0.1) = 5.2527$) for given $\delta = 0.1$.}
	\end{subfigure}
	\end{minipage}\hfill
	\begin{minipage}[t]{0.44\textwidth}
	\begin{subfigure}[t]{\textwidth}
	\centering
	\begin{tikzpicture}[scale = 0.7]
	\begin{axis}[
		xlabel={x},
		ylabel={y},
		width=10cm,
		height=10cm,
		xmin=-22, xmax=50,
		ymin=-22, ymax=50,
		xtick distance=20,
		ytick distance=20
		]
		
		\addplot[
		fill=fmi!30,
		draw=fmi,
		mark=*,
		mark size=1pt,
		mark options={fill=black}
		]
		table [x index=0, y index=1, col sep=comma] {Pictures/Vertices/STEP3_C.csv};
		
		\draw[draw=mydarkred, fill=mydarkred, fill opacity=0.3] 
		(axis cs:0,0) circle (5.5)
		node[below = 7mm, right = 3mm, text=mydarkred, opacity=1] {\large{\textbf{RoI}}};
		
	\end{axis}
\end{tikzpicture}
	\caption{Approximate upper image for $\delta = 0.1$ in the Example \ref{example} with centre of the approximation defined by the chosen RoI $B_{5}((0,0)^T)$.}
	\end{subfigure}
	\end{minipage}\hfill
	\caption{Step 3 of the decision-making procedure. Obtaining a refined impression of the upper image using a low value of $\delta$ ($\delta = 0.1$).}
	\label{fig:step3}
\end{figure}

\begin{figure}[H]
	\centering
	\begin{minipage}[t]{0.48\textwidth}
\begin{tikzpicture}[scale=0.85]
	\begin{axis}[
		xlabel={x},
		ylabel={y},
		width=\textwidth,
		height=\textwidth,
		xmin=-22, xmax=60,
		ymin=-22, ymax=60,
		xtick distance=20,
		ytick distance=20
		]
		
		\addplot[
		fill=fmi!30,
		draw=fmi,
		mark=*,
		mark size=1pt,
		mark options={fill=black}
		]
		table [x index=0, y index=1, col sep=comma] {Pictures/Vertices/STEP4_C.csv};
		
		\draw[draw=mydarkred, fill=mydarkred, fill opacity=0.3] 
		(axis cs:0,0) circle (50.8)
		node[above=22mm, right=32mm, text=mydarkred, opacity=1] {\large{\textbf{RoI}}};
	\end{axis}
\end{tikzpicture}
	\end{minipage}\hfill
	\begin{minipage}[t]{0.48\textwidth}
	\begin{tikzpicture}[scale=0.85]
	\begin{axis}[
		xlabel={x},
		ylabel={y},
		width=\textwidth,
		height=\textwidth,
		xmin=-40, xmax=500,
		ymin=-35, ymax=500,
		xtick distance=100,
		ytick distance=100
		]
		\addplot[
		fill=fmi!30,
		draw=fmi,
		mark=*,
		mark size=1pt,
		mark options={fill=black}
		]
		table [x index=0, y index=1, col sep=comma] {Pictures/Vertices/STEP4_C.csv};
		
		\draw[draw=mydarkred, fill=mydarkred, fill opacity=0.3] 
		(axis cs:0,0) circle (50.8)
		node[above=7mm, right=10mm, text=mydarkred, opacity=1] {\large{\textbf{RoI}}};
	\end{axis}
\end{tikzpicture}
	\end{minipage}
	\caption{Approximate upper image for $\delta = 0.01$ in the Example \ref{example} with centre of the approximation defined by the chosen RoI $B_{50}((0,0)^T)$. Zoomed in (left) and zoomed out (right).}
	\label{fig:step4}
\end{figure}
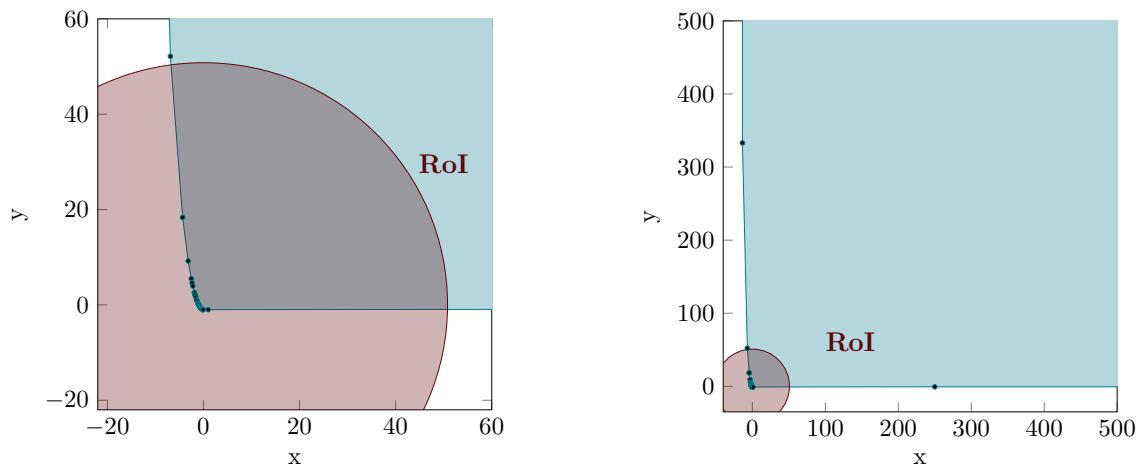
\end{example}
\section{Comparison with Existing Concepts} \label{comparison}
The subsequent section provides a short comparison with the approaches established in the literature.
The majority of the literature on convex vector optimization \cite{lohne2014primal} and convex multi-objective optimization \cite{ehrgott2011approximation} focuses on solution concepts for bounded problems. This assumption, however, is restrictive. Example \ref{example} presents an unbounded and not self-bounded (VCP). Suppose one were to apply a solution concept designed for bounded (VCPs) to a bounded portion of the upper image, that is, to approximate the upper image within a region of interest $B_r(p)$ using tools for bounded approximation. From Figures \ref{fig:step2}, \ref{fig:step3}, and \ref{fig:step4} we can conclude that such an approach yields less global information than a homogeneous $\delta$-approximation. This remains true even if the ordering cone of the (VCP) is added to the bounded approximation. The Figures \ref{fig:step2}, \ref{fig:step3}, \ref{fig:step4} show that vertices of the approximate upper image lying outside the region of interest capture the parabolic shape of the upper image as well as its recession directions. The contrast to the bounded approximation becomes particularly pronounced when the ordering cone is notably smaller than the recession cone of the upper image, e.g. assuming $C \subsetneq \mathbb{R}^m_+$ in Example \ref{example}.\par
When an unbounded problem is not reduced to a bounded one, non‑self‑boundedness of a (VCP) may become an issue. Among the existing notions, the $(\varepsilon ,\delta)$-concept discussed in detail in \cite{wagner2023algorithms} is, to the best of our knowledge, the only framework applicable to general unbounded problems, including non‑self‑bounded ones. However, the homogeneous $\delta$ -solution concept introduced here is more general than the $(\varepsilon ,\delta )$-concept. As Theorem \ref{homExists} establishes, a homogeneous $\delta$-solution exists for any problem of the type (VCP) and any prescribed precision $\delta >0$. In contrast to that, \cite{wagner2023algorithms} do not provide a general result on the existence of $(\varepsilon ,\delta )$-solutions but present an algorithm for finding such a solution. Although they prove the correctness and give some hints on how finiteness of their algorithm might be shown, their approach relies on additional assumptions. In particular, for the validity of the concept, the ordering cone is required to be solid. For the discussion of their algorithm, \cite{wagner2023algorithms} further assume the feasible region to be closed and full-dimensional, as well as the objective function to be continuous on the feasible set. It thereby remains unclear under which general conditions an $(\varepsilon ,\delta )$-solution for (VCP) exists. Beyond that, \cite{wagner2023algorithms} assume polyhedrality of the ordering cone for their algorithmic implementation.\par
Lastly, it shall be noted that within the $(\varepsilon, \delta)$-approximation framework, the radius in which the $\varepsilon$-approximation quality holds cannot be explicitly specified beforehand. In the implementation proposed in \cite{wagner2023algorithms}, this radius depends on the parameters $\varepsilon$ and $\delta$ as well as the properties of the upper image itself. In contrast to that, Section \ref{DMprocedure} demonstrated how the parameter $\delta$ can be chosen in the context of homogeneous $\delta$-approximations so as to guarantee a desired bound on the approximation error within a specified radius. Hence, the decision-making procedure resulting from the homogeneous $\delta$-approximation concept is of a profoundly different nature.
\section{Proofs of Theorems \ref{homExists}, \ref{absch} and \ref{oneOrTwo}}\label{proofs}
In this section, we provide the mathematical proofs for the results given in Section \ref{main}.
\begin{proof}[Proof of Theorem \ref{homExists}] \label{proofhomExists}
	Without loss of generality, let $\delta \in (0,1)$ be given. Note that for $\delta \geq 1$, every finite nonempty set $X \subseteq S$ is a homogeneous $\delta$-solution of (VCP) because the truncated Hausdorff distance of two closed convex cones is always less than or equal to 1.
	Furthermore, $\hom \mathcal{P} \cap \mathbb{B}$ is a nonempty compact convex set. Therefore, according to \cite[Theorem 2.1]{ney1995polyhedral}, there exists a finite set $Y = \{(y^1,\mu^1), \dots, (y^r, \mu^r)\}$ with $r \in \mathbb{N}$ and $ \conv Y \subseteq \hom \mathcal{P} \cap \mathbb{B}$ such that
	\begin{equation}\label{first}
	d_H(\conv Y, \hom \mathcal{P} \cap \mathbb{B}) \leq \frac{\delta}{2}.
	\end{equation}\par
	We can assume without loss of generality that $\dim (\conv Y) = \dim(\hom \mathcal{P} \cap \mathbb{B}) = \dim (\hom \mathcal{P})$ for the following reasoning. The second equality is obvious from the definition. Because $\conv Y \subseteq \hom \mathcal{P} \cap \mathbb{B}$, the relation $\dim(\conv Y) \leq \dim(\hom \mathcal{P} \cap \mathbb{B})$ is always true. If $\dim(\conv Y) < \dim(\hom \mathcal{P} \cap \mathbb{B})$, we can add sufficiently many affinely independent points of $\hom \mathcal{P} \cap \mathbb{B}$ to the set $Y$ until we obtain $\dim( \conv Y) = \dim(\hom \mathcal{P} \cap \mathbb{B})$.\par
	We now proceed to show that we can choose a set of points $\bar{Y} \subseteq \relint(\hom \mathcal{P})$ derived from the given set $Y$ such that $d_{H}(\conv \bar{Y}, \conv Y) \leq \nicefrac{\delta}{2}$. Since $\relint(\conv Y) \neq \emptyset$ \cite[Theorem 6.2]{tyrrell1970convex}, we can choose $(\bar{y}, \bar{\mu})^T \in \relint(\conv Y)$.
	Note that $\bar{\mu} > 0$: It is clear that $\bar{\mu} \geq 0$ since $(\bar{y}, \bar{\mu})^T \in \hom \mathcal{P}$. Since $\mathcal{P} \neq \emptyset$, there exists a point $(z,1)^T \in \hom \mathcal{P}$ and thereby $\hom \mathcal{P} \not\subseteq \{(z,0)^T \mid \:   z \in \mathbb{R}^m\}$. Since $\dim( \conv Y) = \dim(\hom \mathcal{P})$ was assumed, $\conv Y \not \subseteq \{(z,0)^T \mid \:   z \in \mathbb{R}^m \}$ follows. By the characterization of relative interior points in \cite[Theorem 6.4]{tyrrell1970convex} and the fact that $\conv Y \subseteq \hom \mathcal{P} \subseteq \R^m \times \R_+$, this implies that for a relative interior point, the last component $\bar{\mu}$ has to be strictly positive.\par
	For $i = 1, \dots ,r$, define $d^i := (\bar{y}, \bar{\mu})^T - (y^i, \mu^i)^T$. Then $(y^i, \mu^i)^T+ t d^i \in \relint(\conv Y)$ for all $t \in (0,1]$ \cite[Theorem 6.1]{tyrrell1970convex}. Consider
	$$\bar{t} := \frac{\delta}{2 \max\left\{1, \max\left\{\norm{d^i} \; \middle| \; i = 1, \dots, r\right\}\right \}}.$$
	We now can define the set $\bar{Y} = \{ (\bar{y}^1, \bar{\mu}^1)^T, \dots,(\bar{y}^r, \bar{\mu}^r)^T \}$ of relative interior points of $\conv Y$ in the following way.
	$$(\bar{y}^i, \bar{\mu}^i)^T := (y^i, \mu^i)^T + \bar{t} d^i.$$
	Therefore, we have $\conv \bar{Y} \subseteq \relint(\conv Y),$ as well as
	\begin{equation}\label{dHleq}
	d_H(\conv \bar{Y}, \conv Y) \leq \max_{i = 1, \dots, r} \left\lVert (\bar{y}^i, \bar{\mu}^i)^T - (y^i, \mu^i)^T \right\rVert \leq \frac{\delta}{2}.  
	\end{equation}
	Note that the first inequality in Inequality \eqref{dHleq} follows from the fact that the Hausdorff distance between two polytopes is attained at an extreme point \cite[Theorem 3.3]{BATSON1986441}. The second inequality follows from the construction of the set $\bar{Y}$ from the set $Y$. From $\conv \bar{Y} \subseteq \relint(\conv Y)$, we get the relation
	$$\conv \bar{Y} \subseteq \relint(\conv Y)\subseteq \relint(\hom \mathcal{P} \cap \mathbb{B}) \subseteq \relint(\hom \mathcal{P}).$$
	The second and last inclusion follow using the fact that $\dim(\conv Y) = \dim(\hom \mathcal{P} \cap \mathbb{B}) = \dim (\hom \mathcal{P})$ and \cite[Corollary 6.5.2]{tyrrell1970convex}. From Inequalities \eqref{first} and \eqref{dHleq} and the triangle inequality, it then follows that
	\begin{equation} \label{incl1}
		\begin{aligned}
			d_H(\conv \bar{Y}, \hom \mathcal{P} \cap \mathbb{B}) \leq d_H(\conv \bar{Y}, \conv Y) + d_H(\conv Y, \hom \mathcal{P} \cap \mathbb{B}) \leq \delta.
		\end{aligned}
	\end{equation}
	Since $\hom \mathcal{P}$ is a convex cone and $\conv \bar{Y} \subseteq \cone \bar{Y} \cap \mathbb{B}$, we get $\cone \bar{Y} \subseteq \hom \mathcal{P}$ with
	$$d_{tH}(\cone \bar{Y}, \hom \mathcal{P}) \leq \delta.$$
	Thereby, we have shown that, for any $\delta \in (0,1)$, we can find a set $\bar{Y}$ of relative interior points of $\hom \mathcal{P}$ such that a homogeneous $\delta$-approximation of $\mathcal{P}$ can be identified via $\cone \bar{Y}$.\par
	To retrieve a homogeneous $\delta$-approximation and a homogeneous $\delta$-solution for (VCP), consider the following representation for $\relint(\hom \mathcal{P})$ from \cite[Corollary 6.8.1]{tyrrell1970convex}.
	$$\relint(\hom \mathcal{P}) = \left\{\mu (x,1)^T \: \middle| \:   \mu>0, x \in \relint \mathcal{P}\right\}.$$
	Hence, we have 
	\begin{equation} \label{deconification}
		\cone \bar{Y} = \hom\left(\conv\left\{\frac{\bar{y}^1}{\bar{\mu}^1}, \dots, \frac{\bar{y}^r}{\bar{\mu}^r}\right\}\right),
	\end{equation}
	with 
	$$\left\{\frac{\bar{y}^1}{\bar{\mu}^1}, \dots, \frac{\bar{y}^r}{\bar{\mu}^r}\right\} \subseteq \relint \mathcal{P} = \relint (F[S] + C) \subseteq F[S]+C.$$
	Note that $\relint \mathcal{P} = \relint(F[S] + C)$ follows since the relative interior of the closure of the set $F[S]+C$ equals that of the set $F[S]+C$ \cite[Theorem 6.3]{tyrrell1970convex}.
	It is thereby implied that for all $i = 1, \dots, r$:
	$$\frac{\bar{y}^i}{\bar{\mu}^i} = F(x^i) + c^i$$
	for some $x^i \in S$ and some $c^i \in C$. Let $X = \{x^1, \dots, x^r \}$. Then we have 
	\begin{equation} \label{incl2}
		\conv \left\{\frac{\bar{y}^1}{\bar{\mu}^1}, \dots, \frac{\bar{y}^r}{\bar{\mu}^r}\right\} \subseteq \conv F[X] + C \subseteq \mathcal{P}.
	\end{equation}
	Given Inequality \eqref{incl1}, the relation in Equation \eqref{deconification} and Inclusion \eqref{incl2}, we see that $X$ represents a homogeneous $\delta$-solution of the problem (VCP).
\end{proof}

To prove Theorems \ref{absch} and \ref{oneOrTwo}, we first present some supporting statements in Lemmas \ref{nl1}--\ref{nl3}.

\begin{lemma}\label{nl1}
	Let $u, v \in \R^m$ with $\norm{u} = \norm{v} = 1$. Then 
	\begin{equation*}
		d(u, \cone\{v\}) = d(v,\cone\{u\})
		=
		\begin{cases}
			\sqrt{1- (u^Tv)^2},&\text{if } u^Tv \geq 0,\\
			1,&\text{otherwise}.
		\end{cases}
	\end{equation*}
\end{lemma}
\begin{proof}
	We have
	\begin{equation} \label{distPtCone}
		d(u,\cone\{v\})
		=
		\min_{\lambda\geq 0}\|u-\lambda v\|.
	\end{equation}
	Since $\|u-\lambda v\|^2 = (u-\lambda v)^T (u-\lambda v) = 1+\lambda^2-2\lambda u^Tv$, the minimizer in Equation \eqref{distPtCone} is given by $\lambda=\max\{ u^Tv,0\}$.
	Consequently,
	\[
	d(u, \cone\{v\})=
	\begin{cases}
		\sqrt{1- (u^Tv)^2},&\text{if } u^Tv \geq 0,\\
		1,&\text{otherwise}.
	\end{cases}
	\]
	Since the right-hand side is symmetric in $u$ and $v$, it follows that $d(u,\cone\{v\}) = d(v,\cone\{u\})$.
\end{proof}

\begin{lemma}\label{nl2}
	Let $A \subseteq \mathbb{R}^m$ and $B \subseteq \mathbb{R}^m$ be two nonempty closed convex sets with $A \subseteq B$ and $d_{tH}(\hom A, \hom B) \leq \delta$ for some $\delta \in (0,1)$. Then
	\begin{equation*}\label{euclDist}
		\forall a \in \bd(\hom A) \cap \mathbb{B}: d(a, \bd(\hom B) \cap \mathbb{B}) \leq \delta.
	\end{equation*}
\end{lemma}

\begin{proof}
	Let $a \in \bd(\hom A) \cap \mathbb{B}$ and assume, to the contrary, that
	\begin{equation}\label{contrapos}
		d(a, \bd(\hom B) \cap \mathbb{B}) > \delta.
	\end{equation}
	Note that
	\begin{equation}\label{homA}
		d(a,\bd(\hom B)\cap\mathbb B) = d(a,\bd(\hom B)).
	\end{equation}
	Indeed, let $b \in \bd(\hom B)$ realize $d(a,\bd(\hom B))$. Then $b$ is the Euclidean projection of $a$ onto $\cone \{b\} \subseteq \bd(\hom B)$. Since $0 \in \cone \{b\}$ and the projection mapping onto the closed convex set $\cone\{b\}$ is nonexpansive, it follows that $\norm{b} \leq \norm{a} \leq 1$. Hence, $b \in \bd(\hom B)\cap\mathbb B$.
	
	Since $a \in \bd(\hom A)$, and $\hom A$ is closed and convex, there exists a supporting hyperplane to $\hom A$ in $a$ \cite[p. 51]{Boyd2004}, i.e.,
	\begin{equation}\label{suppH}
		\exists w \in \R^{m+1}: \; \norm{w} = 1, \quad \forall z \in \hom A: \; w^Tz \leq 0, \quad w^Ta = 0.
	\end{equation}
	Because of Inequalities \eqref{contrapos} and \eqref{homA} and $\delta < 1$ we can choose $s >0$ such that
	\begin{equation*}
		\delta < s < \min\left\{d(a, \bd(\hom B)), 1\right\}.
	\end{equation*}
	Define $t:= \sqrt{1-s^2}a+sw$. We show that $t \in \hom B$. Define $p:= \sqrt{1-s^2}t = (1-s^2)a + \sqrt{1-s^2}sw$. Since $\norm{a} \leq 1$, $\norm{w} = 1$ and $w^Ta = 0$, we get $\norm{p-a}^2 \leq s^2$, which implies $\norm{p-a} \leq s < d(a, \bd(\hom B))$. With $a \in \hom A \subseteq \hom B$ and the convexity and closedness of $\hom A$ and $\hom B$, this implies $p \in \hom B$. Because $\hom B$ is a cone, $t \in \hom B$ follows.
	
	Using the definition of $t$ and the orthogonality of $a$ and $w$, we get
	$$\norm{t}^2 = (1-s^2) \norm{a}^2 + s^2 \norm{w}^2 \leq 1.$$ 
	Thus, $t \in \hom B \cap \mathbb{B}$. Now, it follows from the Cauchy-Bunyakovsky-Schwarz Inequality that
	\begin{equation} \label{eqZ}
		\forall z \in \hom A \cap \mathbb{B}: \; (t-z)^Tw \leq \norm{t-z}\norm{w}.
	\end{equation}
	We have $\norm{w} = 1$ and with the definition of $t$ as well as Equation \eqref{suppH}
	\begin{equation*}
		(t-z)^Tw = w^Tt - z^Tw = \sqrt{1-s^2}w^Ta + s \norm{w}^2 - w^Tz \geq s.
	\end{equation*}
	Together with Inequality \eqref{eqZ}, this implies
	\begin{equation*}
		\forall z \in \hom A \cap \mathbb{B}: \; \norm{t-z} \geq s > \delta.
	\end{equation*}
	Thus, $d(t, \hom A \cap \mathbb{B}) > \delta$ with $t \in \hom B \cap \mathbb{B}$. This is a contradiction to the assumption $d_{tH}(\hom A, \hom B) \leq \delta$.
\end{proof}

\begin{lemma} \label{lemma} Let $A \subseteq \mathbb{R}^m$ and $B \subseteq \mathbb{R}^m$ be two nonempty closed convex sets with $A \subseteq B$ and $d_{tH}(\hom A, \hom B) \leq \delta$ for some $\delta \in (0,1)$. Let $y \in \bd A$ with $\lVert y \rVert < R_\delta :=\sqrt{\nicefrac{1}{\delta ^2}-1}$. Then
	$$\exists n \in \bd B: d\left(\frac{(n,1)^T}{\left \lVert(n,1)^T\right \rVert}, \cone \left\{(y,1)^T\right\} \right) \leq \delta.$$
\end{lemma}

\begin{proof} 
	Consider $y \in \bd A$ with $\lVert y \rVert < R_\delta$. From Lemma \ref{nl1}, we obtain for arbitrary $\bar{n} \in \mathbb{R}^{m+1}\setminus \{0\}$ that
	\begin{equation} \label{geomarg}
		d\left(\frac{\bar{n}}{\left \lVert\bar{n}\right \rVert}, \cone \left\{(y,1)^T\right\} \right) = d\left(\frac{(y,1)^T}{\left \lVert(y,1)^T\right \rVert}, \cone \left\{\bar{n}\right\} \right). 
	\end{equation}
	From the definition of the homogenization of a closed convex set, it follows that $y \in \bd A$ implies $(y,1)^T \in \bd(\hom A)$. For the point $\nicefrac{(y,1)^T}{\norm{(y,1)^T}} \in \bd(\hom A)\cap \mathbb B$, Lemma \ref{nl2} and the compactness of $\bd(\hom B) \cap \mathbb{B}$ imply that there exists a point $\bar n \in \bd(\hom B)$ with
	\begin{equation}\label{eq12}
		d\left(\frac{(y,1)^T}{\norm{(y,1)^T}}, \cone\left\{\bar n\right\}\right) \leq \delta.
	\end{equation}
	Since $\delta<1$, we have $\bar n\neq 0$. With Equation \eqref{geomarg}, this implies
	\begin{equation*} \label{eqLast}
		d\left(\frac{\bar{n}}{\left \lVert\bar{n}\right \rVert}, \cone \left \{(y,1)^T \right \} \right) \leq \delta.
	\end{equation*} \par
	It remains to show that $\bar{n}$ can be chosen of the form $\bar{n} = (\tilde{n},1)^T$. Since $\bar{n}$ is contained in the homogenization of a nonempty convex set, it is of the form $\bar{n} = (\tilde{n},\mu)^T$ with $\mu \geq 0$. Clearly, if $\mu > 0$, 
	$$\frac{\bar{n}}{\left \lVert\bar{n}\right \rVert} = \frac{(n,1)^T}{\left \lVert(n,1)^T\right \rVert} \quad \text{for } n = \frac{\tilde{n}}{\mu}$$ 
	and the desired statement holds for $n \in \bd B$.
	
	Assume $\mu = 0$ was to hold. Consider the orthogonal projection $\pi_H$ onto $H$ with:
	\[
	H := \R^m \times \{0\},
	\qquad 
	\pi_H : \mathbb{R}^{m+1}\to H, \qquad \pi_H((z,k)^T) = (z,0)^T \;\; \text{for } z \in \R^m
	\]
	Then $\pi_H ((y,1)^T) = (y,0)^T$ and we have 
	$$d\left(\frac{(y,1)^T}{\norm{(y,1)^T}}, H\right) = d\left(\frac{(y,1)^T}{\norm{(y,1)^T}}, \frac{(y,0)^T}{\norm{(y,1)^T}}\right) = \frac{1}{\sqrt{1+\norm{y}^2}}.$$
	Since $\bar n \in H$, we thereby obtain the contradiction
	\begin{equation*}
		\delta \geq d\left(\frac{(y,1)^T}{\norm{(y,1)^T}}, \cone\left\{\bar{n}\right\}\right) \geq d\left(\frac{(y,1)^T}{\norm{(y,1)^T}}, H\right) = \frac{1}{\sqrt{1+\norm{y}^2}} > \delta,
	\end{equation*}
	where the first inequality follows from Inequality \eqref{eq12} and the last inequality follows from the assumption $\norm{y} < R_\delta$.
\end{proof}

\begin{lemma} \label{nl3}
	Let $y \in \R^m$ and $n \in \R^m$ with
	$$q:= d\left(\frac{(y,1)^T}{\norm{(y,1)^T}}, \cone\left\{(n,1)^T\right\}\right) \text{ and } \sqrt{1-q^2} - q \norm{y} > 0.$$
	Then
	$$\norm{n-y} \leq \frac{q\left(\norm{y}^2 + 1\right)}{\sqrt{1-q^2} - q \norm{y}} =: \alpha(\norm{y}, q).$$
\end{lemma}
\begin{proof}
	If $q = 0$, then $(y,1)^T= (n,1)^T$ and the statement holds trivially in this case. Assume $q > 0$ from this point onward. We define:
	\begin{equation*}
		z_1 := \frac{\sqrt{1-q^2}}{\left \lVert(y,1)^T\right \rVert} (y,1)^T, \quad
		z_2 := \frac{(n,1)^T}{\left \lVert(n,1)^T\right \rVert}.
	\end{equation*}
	Note that, by definition, $z_1$ belongs to $\cone\{(y,1)^T\}$ and $z_2$ to $\cone\{(n,1)^T\}$ with
	\begin{equation*}
		\left \lVert z_1\right \rVert = \sqrt{1-q^2}, \quad \left \lVert z_2 \right \rVert = 1.
	\end{equation*}
	Because of the assumption $\sqrt{1-q^2} - q \norm{y} > 0$, it follows that $0 < q< 1$. Hence, by the definition of $q$ and Lemma \ref{nl1}, $z_1$ is exactly the orthogonal projection of $z_2$ onto the ray  $\cone\{(y,1)^T\}$. Therefore, we have
	\begin{equation} \label{z2Wz1}
		z_2 = z_1 + qw
	\end{equation}
	for some $w:= (a,b)^T$ with $a \in \R^m, b \in \R$, $\norm{w} = 1$ and $w\perp (y,1)^T$. Since $w \perp (y,1)^T$, 
	\begin{equation}\label{defB}
		b = -a^Ty.
	\end{equation}
	Since $\norm{w} = 1$,
	\begin{equation}\label{defBsq}
		\norm{a}^2 = 1 - b^2.
	\end{equation} 
	From Equations \eqref{defB} and \eqref{defBsq} and the Cauchy-Bunyakovsky-Schwarz Inequality we derive
	\begin{equation*}\label{cs}
		b^2 \leq \norm{a}^2 \norm{y}^2 = (1-b^2) \norm{y}^2.
	\end{equation*}
	Thus, if we define $c:= \lVert (y,1)^T \rVert b$, we get 
	\begin{equation} \label{cleqy}
		\lvert c \rvert \leq \norm{y}.
	\end{equation}
	As $\norm{z_2} = 1$ and $\lVert (n,1)^T \rVert \geq 1$, there exists $k > 0$ such that $(n,1)^T = k z_2$. Thus, Equation \eqref{z2Wz1} yields
	\begin{equation*}
		(n,1)^T = k \left(\frac{\sqrt{1-q^2}}{\left \lVert(y,1)^T\right\rVert} (y,1)^T+ q (a,b)^T\right).
	\end{equation*}
	Considering the first $m$ coordinates and the last coordinate separately, this implies that
	\begin{equation} \label{equation1}
		1 = k\left(\frac{\sqrt{1-q^2}}{\left \lVert(y,1)^T\right \rVert}+qb\right) \Longrightarrow k = \frac{\left \lVert(y,1)^T\right \rVert}{\sqrt{1-q^2}+qc},
	\end{equation}
	as well as
	\begin{equation} \label{equation2}
		n = k \left(\frac{\sqrt{1-q^2}}{\left \lVert(y,1)^T\right \rVert} y + qa\right).
	\end{equation}
	Combining Equations \eqref{equation1} and \eqref{equation2} yields
	\begin{equation*} \label{equation3}
		n = \frac{\left \lVert(y,1)^T\right \rVert}{\sqrt{1-q^2}+qc}\left(\frac{\sqrt{1-q^2}}{\left \lVert(y,1)^T\right \rVert} y + qa\right).
	\end{equation*}
	Thus, 
	\begin{equation*} \label{nMinusy}
		n - y = \frac{q\left(\norm{(y,1)^T} a - cy\right)}{\sqrt{1-q^2} + qc}
	\end{equation*}
	and with Equations \eqref{defB} and \eqref{defBsq} and some calculation we derive
	\begin{equation} \label{nMinusyNorm}
		\norm{n - y} = \left\lvert \frac{q}{\sqrt{1-q^2} + qc}\right\rvert \norm{(y,1)^T} \sqrt{1+c^2}.
	\end{equation}
	Note that Inequality \eqref{cleqy} together with our assumptions yields
	\begin{equation}\label{denominator}
		\sqrt{1-q^2} + qc \geq \sqrt{1-q^2} - q \norm{y} > 0.
	\end{equation}
	Thus, Equation \eqref{nMinusyNorm} can be written as
	\begin{equation}\label{eq21}
		\norm{n - y} = \frac{q\norm{(y,1)^T} \sqrt{1+c^2}}{\sqrt{1-q^2} + qc}.
	\end{equation}
	Inequality \eqref{cleqy} gives us
	$$\sqrt{1+c^2} \leq \norm{(y,1)^T}.$$
	Together with Inequality \eqref{denominator} and Equation \eqref{eq21}, this yields
	\begin{equation*}\label{normLeqQ}
		\norm{n-y} \leq \frac{q\left(\norm{y}^2 + 1\right)}{\sqrt{1-q^2} - q \norm{y}} = \alpha(\norm{y},q).
	\end{equation*}
\end{proof}

\begin{proof}[Proof of Theorem \ref{absch}] Due to the assumption $r<R_\delta$  and $P_X$ being an inner homogeneous $\delta$-approximation of $\mathcal{P}$, Lemma \ref{lemma} states that there exists $(n,1)^T \in \bd(\hom \mathcal{P})$ such that
$$q := d\left(\frac{(n,1)^T}{\left \lVert(n,1)^T\right \rVert} , \cone\left\{(y,1)^T\right\}\right) \leq \delta.$$
If $q = 0$, then $(n,1)^T= (y,1)^T$ and thus $y \in \bd \mathcal P$. Because $r < R_\delta$, $\alpha(r, \delta) > 0$. Hence, the statement of Theorem \ref{absch} holds trivially in this case. Assume $q > 0$ from this point onward. 

Applying Lemma \ref{nl3}, we obtain
\begin{equation}\label{eq22}
	d(y, \bd \mathcal P) \leq \norm{n-y} \leq \frac{q\left(\norm{y}^2 + 1\right)}{\sqrt{1-q^2} - q \norm{y}} = \alpha(\norm{y},q).
\end{equation}
Note that the condition $\sqrt{1-q^2}- q \norm{y} > 0$ of Lemma \ref{nl3} is satisfied because we assume $\norm{y} < R_\delta$ and $q \leq \delta$. Hence, $q \leq \delta < \nicefrac{1}{\sqrt{1+\norm{y}^2}}$. The first derivative of $\alpha(\norm{y},q)$ in $q$ is 
$$\frac{\partial \alpha(\norm{y},q)}{\partial q} = \frac{\norm{y}^2+1}{\sqrt{1-q^2}\left(\sqrt{1-q^2}-q\norm{y}\right)^2},$$ 
which is strictly positive wherever it is defined for $q \in (0,1)$. Hence, the function is strictly increasing in $q$ on the interval $0 < q < \nicefrac{1}{\sqrt{1+\norm{y}^2}}$. Next, consider the first derivative in $\norm{y}$.
$$\frac{\partial \alpha(\norm{y},q)}{\partial \norm{y}} = \frac{q\left(q+2\sqrt{1-q^2}\norm{y}-q\norm{y}^2\right)}{\left(q\norm{y}-\sqrt{1-q^2}\right)^2}.$$
Under the assumption $0 < q < \nicefrac{1}{\sqrt{\norm{y}^2+1}}$, the derivative is strictly positive and thus the function is monotonically increasing in $\norm{y}$. Therefore, given $r < R_\delta$, and the fact that $q \leq \delta$, it follows that 
\begin{equation*} \label{abschProof}
	\alpha(\norm{y},q) = \frac{q(\norm{y}^2+1)}{\sqrt{1-q^2}-q\norm{y} } \leq \frac{\delta(r^2+1)}{\sqrt{1-\delta^2}-\delta r } = \alpha(r,\delta).
\end{equation*}
Together with Inequality \eqref{eq22}, this yields the statement of Theorem \ref{absch}.
\end{proof}

\begin{proof}[Proof of Theorem \ref{oneOrTwo}]
	By Lemma \ref{nl2} and the compactness of $\bd(\hom \mathcal P) \cap \mathbb{B}$, there exists $\bar n \in \bd(\hom \mathcal P) \cap \mathbb{B}$ such that
	\[
	d\left(\frac{(y,1)^T}{\|(y,1)^T\|},\bar n \right) \leq \delta.
	\]
	Thus,
	\[
	d\left(\frac{(y,1)^T}{\|(y,1)^T\|},\cone\{\bar n \}\right) \leq \delta.
	\]
	Since $\delta < 1$, we have $\bar n \neq 0$. Using Lemma \ref{nl1}, we obtain
	\[
	d\left(\frac{\bar n}{\|\bar n\|},\cone\left\{(y,1)^T\right\}\right) \leq \delta.
	\]
	Hence, by the characterization in Proposition \ref{rock}, one of the following holds true.
	\begin{enumerate}[label=(\roman*)]
		\setcounter{enumi}{2}
		\item \label{itm: iii} $\exists (n,1)^T \in \bd(\hom \mathcal{P}): d\left(\frac{(n,1)^T}{\left \lVert(n,1)^T\right \rVert}, \cone\left\{(y,1)^T\right\}\right) \leq \delta.$
		\item \label{itm: iv} $\exists (d,0)^T \in \bd(\hom \mathcal{P}): d\left(\frac{(d,0)^T}{\left \lVert(d,0)^T\right \rVert}, \cone\left\{(y,1)^T\right\}\right) \leq \delta.$
	\end{enumerate}
	We first show that \ref{itm: iii} implies \ref{itm:i} or \ref{itm: ii}. Let $(n,1)^T \in \bd (\hom \mathcal P)$ be the point from \ref{itm: iii} satisfying
	\begin{equation}\label{ass}
		q:=d\left(\frac{(n,1)^T}{\left\lVert (n,1)^T \right\rVert}, \cone\left\{(y, 1)^T\right\}\right) \leq \delta.
	\end{equation}
	First, assume that $q < \nicefrac{1}{\sqrt{1+\norm{y}^2}}$. This is the case if and only if $\sqrt{1-q^2} - q \norm{y} > 0$. Thus, Lemma \ref{nl3} yields $d(y,\bd\mathcal P)\leq \norm{n-y} \leq \alpha(\norm{y},q)$. This directly implies \ref{itm:i} for $n \in \bd \mathcal P$. Assume now that 
	\begin{equation}\label{assCase2}
		q \geq \frac{1}{\sqrt{1+\norm{y}^2}}.
	\end{equation}
	Since $q \leq \delta < 1$, this means that $y \neq 0$ and also $\nicefrac{1}{\sqrt{1+\norm{y}^2}} \leq \delta$. If, in addition, $y \in 0^+\mathcal P$, then
	$$d\left(\frac{y}{\norm{y}}, 0^+\mathcal P\right) = 0,$$
	which implies \ref{itm: ii}. Consider $y \notin 0^+ \mathcal P$. Since $y \in P_X \subseteq \mathcal P$ and $\mathcal P$ is a nonempty closed convex set with $y \notin 0^+\mathcal P$, \cite[Theorem 8.3]{tyrrell1970convex} implies that there exists $t^* := \max\{t \geq 1 \mid \; ty \in \mathcal P\}$ corresponding to a point $n^* := t^* y \in \bd \mathcal P$. Then
	$$q^* := d\left(\frac{(n^*,1)^T}{\left\lVert (n^*,1)^T \right\rVert}, \cone\left\{(y, 1)^T\right\}\right) = \sqrt{1-\left[\left(\frac{(n^*,1)^T}{\left\lVert (n^*,1)^T \right\rVert}\right)^T\left(\frac{(y,1)^T}{\left\lVert (y,1)^T \right\rVert}\right)\right]^2 }.$$
	This follows from Lemma \ref{nl1} and the fact that 
	$$\left(\frac{(n^*,1)^T}{\left\lVert (n^*,1)^T \right\rVert}\right)^T\left(\frac{(y,1)^T}{\left\lVert (y,1)^T \right\rVert}\right) = \frac{t^* \norm{y}^2 + 1}{\sqrt{t^{*2}\norm{y}^2 + 1}\sqrt{\norm{y}^2 + 1}} > 0.$$
	Thus, some further calculation and $t^* \geq 1$ yield
	$$q^* = \frac{(t^* - 1) \norm{y}}{\sqrt{1+\norm{y}^2}\sqrt{1+t^{*2} \norm{y}^2}}.$$
	Since $t^* \geq 1$ and $t^{*2}\norm{y}^2 < 1 + t^{*2} \norm{y}^2$, we have $(t^* - 1) \norm{y} \leq t^* \norm{y} < \sqrt{1 + t^{*2} \norm{y}^2}$, i.e.,
	$$q^* < \frac{1}{\sqrt{1+ \norm{y}^2}} \leq \delta.$$
	The last inequality follows from our assumptions in Equations \eqref{ass} and \eqref{assCase2}. Since this means $\sqrt{1-(q^*)^2}-q^*\norm{y}>0$, Lemma \ref{nl3} yields $d(y,\bd\mathcal P)\leq \norm{n^* -y} \leq \alpha(\norm{y},q^*)$. This directly implies \ref{itm:i} for
	$n^* \in \bd \mathcal P$.

	We now show that \ref{itm: iv} implies \ref{itm: ii}. Let \ref{itm: iv} hold. By Proposition \ref{rock}, this implies that there exists $d \in 0^+\mathcal{P}$ such that
	\begin{equation*}
		d\left(\frac{(d,0)^T}{\left\lVert (d,0)^T \right\rVert}, \cone\left\{(y, 1)^T\right\}\right) \leq \delta.
	\end{equation*}
	With Lemma \ref{nl1}, we have
	\begin{equation*} \label{eqLeqQ}
		q:= d\left(\frac{(y,1)^T}{\left\lVert (y,1)^T \right\rVert}, \cone\left\{(d, 0)^T\right\}\right) \leq \delta.
	\end{equation*}
	Note that $y \neq 0$, since otherwise $q=1$, contradicting $q\leq\delta<1$. For the same reasoning, we can assume without loss of generality that $\norm{d} = 1$. The closest point to $\nicefrac{(y,1)^T}{\lVert (y,1)^T \rVert}$ on the ray $\cone\{(d,0)^T\}$ is its orthogonal projection $k (d,0)^T$ for some $k > 0$, hence
	\[
	q = \left\lVert \frac{(y,1)^T}{\left\lVert (y,1)^T \right\rVert} - k (d,0)^T\right\rVert.
	\]
	Since $\nicefrac{(y,1)^T}{\lVert (y,1)^T \rVert}$ is a unit vector, the Pythagorean identity gives $1 = k^2 + q^2$. Thus,
	\begin{equation*}
	q =
	\left \lVert
	\frac{(y,1)^T}{\lVert (y,1)^T\rVert} - 
	\sqrt{1-q^2}\,(d,0)^T
	\right \rVert
	\end{equation*}
	Since omitting the last coordinate cannot increase the Euclidean norm and $\sqrt{1-q^2}\,d \in \cone\{d\}$, we obtain
	$$q \geq \left\lVert \frac{y}{\left\lVert (y,1)^T \right\rVert}
		- \sqrt{1-q^2}\,d \right\rVert \geq d\left(\frac{y}{\left\lVert (y,1)^T \right\rVert},\cone\{d\}\right) =: a.$$
	Since $\cone\{d\}$ is a cone, it is invariant under positive scaling. Hence, for any $\lambda >0$ we have
	\begin{equation}\label{coneEucl}
	d(\lambda y,\cone\{d\}) = \lambda\, d(y,\cone\{d\}),
	\end{equation}
	because
	\begin{equation*}
	d(\lambda y,\cone\{d\})
	= \inf_{z\in\cone\{d\}} \|\lambda y - z\|
	= \inf_{w\in\cone\{d\}} \|\lambda(y - w)\|
	= \lambda\, d(y,\cone\{d\}).
	\end{equation*}
	Using Equation \eqref{coneEucl} as well as $a \leq q \leq \delta$, we get the following inequality.
	\begin{equation}\label{eqB}
		b:= d\left(\frac{y}{\left\lVert y \right\rVert}, \cone\{d\}\right) = \frac{\left\lVert (y,1)^T \right\rVert}{\left\lVert y \right\rVert} a \leq \frac{\left\lVert (y,1)^T \right\rVert}{\left\lVert y \right\rVert}q \leq \frac{\left\lVert (y,1)^T \right\rVert}{\left\lVert y \right\rVert} \delta.
	\end{equation}
	Since $d \in 0^+ \mathcal{P}$, we get
	\begin{equation}\label{eqBgeq}
	b \geq d\left(\frac{y}{\left\lVert y \right\rVert}, 0^+ \mathcal{P}\right).
	\end{equation}
	Hence, from Inequalities \eqref{eqB} and \eqref{eqBgeq}, we derive 
	$$d\left(\frac{y}{\left\lVert y \right\rVert}, 0^+ \mathcal{P}\right) \leq \frac{\left\lVert (y, 1)^T \right\rVert}{\left\lVert y \right\rVert} \delta$$
	and thereby statement \ref{itm: ii} is implied. Since either \ref{itm: iii} or \ref{itm: iv} holds, and since we have shown that \ref{itm: iii} implies \ref{itm:i} or \ref{itm: ii}, while \ref{itm: iv} implies \ref{itm: ii}, the assertion of Theorem \ref{oneOrTwo} follows.
\end{proof}
\section{Conclusions} \label{outlook}
In this work, we introduced the concept of homogeneous $\delta$-solutions for convex vector optimization problems. This notion relies on a single precision parameter $\delta$, avoids distinguishing between points far from the origin and directions, and is robust under scaling of the objective functions. We established the existence of a homogeneous $\delta$-solution for any problem of type (VCP) and any $\delta>0$, without additional structural assumptions. Moreover, a form of approximate minimality is guaranteed: boundary points of the approximation within a user-defined RoI of radius $r$ are ensured to lie within a distance of at most $\alpha(r,\delta)$ to the boundary of the upper image and thus to a weakly $C-$minimal element given feasibility. We provided an explicit expression for $\alpha(r,\delta)$ as a function of $r$ and $\delta$. Outside the RoI, boundary points of the approximation exhibit proximity to boundary points or to recession directions of the upper image.\par
We further demonstrated a decision-making procedure that highlights the practical relevance of the concept. A decision-maker selects the precision parameter $\delta$  by balancing the trade-off between the size of the RoI and the guaranteed approximation quality, keeping $\delta$  as large as possible to enhance computational efficiency. The procedure can be iteratively refined, allowing stepwise exploration of the upper image with adaptable RoI and precision.\par
Finally, we contrasted the homogeneous $\delta$-approximation with existing approaches. In contrast to bounded approaches, homogeneous $\delta$-approximations offer substantial global information. Compared to the $(\varepsilon ,\delta )$-framework, which is to the best of our knowledge the only framework applicable to general unbounded problems, our approach requires no additional assumptions and guarantees the existence of a solution. Beyond the technical distinctions, the proposed homogeneous $\delta$-concept offers a fundamentally different perspective: it combines global information with user-specified step-wise local refinement. Thereby we provide a versatile framework for decision-making in convex vector optimization.

\medskip 

\subsection*{Disclosure of interest}
The authors report there are no competing interests to declare.

\subsection*{Declaration of Generative AI Use}
During the preparation and revision of this manuscript, the authors used ChatGPT (OpenAI, GPT-6.2) to improve language, readability, and grammatical accuracy. The authors reviewed and verified all AI-assisted changes and confirmed the accuracy of the content. The authors take full responsibility for the integrity and accuracy of the manuscript.

\newpage
\sloppy
\printbibliography[heading=bibintoc, notcategory=uncited]
\fussy

\end{document}